\documentclass[reqno,11pt,a4paper]{amsart}

\usepackage{amsthm,amsmath,bm,amssymb,extarrows,amsfonts,mathrsfs,comment}
\numberwithin{equation}{section}
\usepackage{enumerate}
\usepackage{enumitem}
\usepackage{tikz-cd,tikz}
\usepackage[compat=1.1.0]{tikz-feynhand}
\usetikzlibrary{intersections, calc, arrows.meta,positioning,shapes.geometric,automata,shapes.misc}
\usepackage{caption}
\usepackage{accents,setspace}
\usepackage{graphicx}

\theoremstyle{definition}
\newtheorem{dfn}{Definition}[section]

\newtheorem{thm}[dfn]{Theorem}
\newtheorem{lem}[dfn]{Lemma}
\newtheorem{prop}[dfn]{Proposition}
\newtheorem{cor}[dfn]{Corollary}
\newtheorem{rem}[dfn]{Remark}
\newtheorem{ex}[dfn]{Example}

\title[HMS for WPS and Morse Homotopy]{Homological mirror symmetry for weighted projective spaces and Morse Homotopy}
\date{\today}
\author{Azuna Nishida}
\address{Department of Mathematics and Informatics, Graduate School of Science and Engineering, Chiba University, Yayoicho 1-33, Inage, Chiba, 263-8522 Japan}
\email{anishida@g.math.s.chiba-u.ac.jp}

\begin{document}

\begin{abstract}
Kontsevich and Soibelman discussed homological mirror symmetry 
by using the SYZ torus fibrations, where they introduced the weighted version of
 Fukaya-Oh's Morse homotopy on the base space of
the dual torus fibration in the intermediate step. 
Futaki and Kajiura applied Kontsevich-Soibelman's approach to the case when a complex manifold $X$ is a smooth compact toric
manifold. 
There, they introduced the category of weighted Morse homotopy on the moment polytope of toric manifolds, 
and compared this category to the derived category of coherent sheaves on $X$ instead of the Fukaya category.
In this paper, we extend their setting to the case of toric orbifolds, and discuss this version of homological mirror symmetry for weighted projective spaces.

\end{abstract}

\maketitle
\tableofcontents

\section{Introduction}
In \cite{kontsevich1995homological}, Kontsevich proposed a categorical formulation of mirror symmetry for Calabi-Yau manifolds.   
Since then, beyond the original setting of mirror pairs of Calabi-Yau manifolds, Kontsevich's homological mirror symmetry (HMS) conjecture has been studied for a larger class of mirror pairs with some adjustments in setting. 
For Fano manifolds and their Landau-Ginzburg mirrors, HMS is originally discussed as an equivalence of triangulated categories between the bounded derived category of coherent sheaves on a toric variety and the derived Fukaya-Seidel category of the Lefschetz fibration defined by the Landau-Ginzburg potential.  
   In this version of HMS, many works have been done, for example, for some (toric) Fano varieties or stacks \cite{seidel2001vanishing, seidel2001more,ueda2006homological,auroux2006mirror,auroux2008mirror}.
   Notably, Auroux-Katzarkov-Orlov proved HMS for weighted projective planes (and their noncommutative deformations) in \cite{auroux2008mirror}. 
   Different versions of HMS are discussed, for example, by Abouzaid \cite{abouzaid2009morse}, by Fang\cite{fang2008homological}, by Fang-Liu-Treumann-Zaslow \cite{fang2014coherent} and by Kuwagaki \cite{kuwagaki2020nonequivariant}. 
   
   Recently, Futaki-Kajiura \cite{futaki2021homological} proposed a way of 
   understanding homological mirror symmetry for smooth compact toric manifolds from the Strominger-Yau-Zaslow's (SYZ) viewpoint. 
   Strominger-Yau-Zaslow \cite{strominger1996mirror} gave a geometric interpretation of mirror symmetry and proposed a construction of mirror pairs as torus fibrations which are fiberwise dual to each other over the same base.
   Kontsevich-Soibelman \cite{kontsevich2001homological} discussed homological mirror symmetry along this line via the Fukaya-Oh category for the dual torus fibration, which we also call the category of weighted Morse homotopy.  
   Here, Morse homotopy was first introduced by Fukaya in \cite{fukaya1993morse}.   
   Fukaya-Oh \cite{fukaya1997zero} proved that the category of Morse homotopy on a closed manifold is equivalent to the Fukaya category of its cotangent bundle.  
    Futaki-Kajiura introduced the category $Mo(P)$ of weighted Morse homotopy on the moment polytope of compact toric manifolds as a generalization of the category of the weighted Morse homotopy to the case where the base manifold has boundaries. 
    They used this category as a substitute of the Fukaya category of a mirror of a smooth toric manifold $X$, and proposed a version of HMS as an equivalence of the form 
    \[Tr(Mo_{\mathcal{E}}(P))\simeq D^b(coh(X)),\]
   where $\mathcal{E}$ denotes the finite set of Lagrangian sections that are SYZ mirror to holomorphic line bundles in a chosen full exceptional collection of $D^b(coh(X))$. 
   We denote by $Mo_{\mathcal{E}}(P)$ the full subcategory of $Mo(P)$ consisting of Lagrangian sections in $\mathcal{E}$ and $Tr$ denotes Bondal-Kapranov-Kontsevich's construction of triangulated categories \cite{bondal1990enhanced,kontsevich1995homological}.
   This formulation works at least when there exists a full exceptional collection consisting of line bundles on $X$. 
   So far, the equivalence above has been shown when $X$ is the projective space $\mathbb{P}^n$ and their product $\mathbb{P}^n\times \mathbb{P}^m$ in \cite{futaki2021homological}, the first Hirzebruch surface $\mathbb{F}_1$ in \cite{futaki2022homological}, the remaining two cases of toric Fano surfaces by Nakanishi \cite{nakanishi2024homological} and the Hirzebruch surface $\mathbb{F}_k$ with $k\geq 1$ in \cite{nakanishi2024syz}. 
   
    In this paper, we discuss an extension of this version of HMS set-up to compact toric orbifolds, and show that the homological mirror symmetry in the above sense holds for the weighted projective space $\mathbb{P}(q_0,\ldots,q_n)$ with $\gcd(q_0,\ldots,q_n)=1$. Explicitly, we show that there exists an equivalence of triangulated categories
    \[Tr(Mo_{\mathcal{E}}(P))\simeq D^{b}(coh(\mathbb{P}(q_0,\ldots,q_n))).\]
    See Corollary \ref{HMSforWPS}.
   
   This paper is organized as follows.
   In section 2, we recall some definitions of (effective) orbifolds or $V$-manifolds and orbifold vector bundles, which we need for the geometric SYZ approach. 
   In section 3, we also recall toric orbifolds associated to stacky fans 
   and holomorphic line bundles on them, which gives us a combinatorial description of the weighted projective spaces. When we consider the derived category of coherent sheaves on the weighted projective space, we also treat it as a stack.
    In this paper, we refer to a smooth toric Deligne-Mumford stack with generically trivial stabilizers as a toric orbifold.
   In section 4, we discuss an extension of the SYZ torus fibrations set-up to toric orbifolds, and demonstrate it in the case of the weighted projective spaces. 
   In section 5, we recall categories on the both sides for homological mirror symmetry in the above sense. 
   In section 6, we compute the  full subcategory $Mo_{\mathcal{E}}(P)$ of the category $Mo(P)$ of weighted Morse homotopy on the polytope $P$ of the weighted projective space and show the main theorem (Theorem \ref{maintheorem}).
   
   \vspace{\baselineskip}
   \noindent\textbf{Acknowledgements.}
   The author is grateful to the advisor, Hiroshige Kajiura, for sharing his insights and for valuable advice. 
   The author would also like to thank Manabu Akaho for discussion on symplectic orbifolds, and Shinnosuke Okawa and Takahiro Tsushima for explaining the structure of the complex side. 
   The author is also grateful to Masahiro Futaki, Hayato Nakanishi, Kentaro Yamaguchi, and Yukiko Konishi for helpful discussions and for valuable comments.  
   This work was supported by JST SPRING, Grant Number JPMJSP2109.

\section{Preliminaries on orbifolds}
In order to fix notations, we briefly recall and collect some facts about orbifolds or $V$-manifolds in the sense of Satake \cite{satake1956generalization} and orbifold vector  bundles. For more details, we refer the readers to \cite{boyer2007sasakian,satake1957gauss,chen2002orbifold,moerdijk1997orbifolds, baily1957imbedding,adem2007orbifolds}, and the references therein.

\subsection{Orbifolds and  orbifold vector bundles}
Let $X$ be a paracompact Hausdorff space. An $n$-dimensional  \textit{orbifold chart} for an open  set $U\subset X$ is a triple $(\tilde{U},\Gamma, \varphi)$, where $\tilde{U}$ is a connected open subset of $\mathbb{C}^n$, $\Gamma$ is a finite group acting holomorphically and effectively
on $\tilde{U}$, and $\varphi:\tilde{U}\rightarrow U$ is a $\Gamma$-invariant continuous surjective map   
    such that 
    the induced map of $\tilde{U}/\Gamma$ onto $\tilde{U}$ is a biholomorphic.
    
    Let $(\tilde{U},\Gamma,\varphi),\ (\tilde{U'},\Gamma',\varphi')$ be orbifold charts for open sets $U,\ U'$, respectively, and let $U\subset U'$.  
    An \textit{injection} $\lambda:(\tilde{U},\Gamma,\varphi)\rightarrow (\tilde{U'},\Gamma',\varphi')$ is a holomorphic embedding $\lambda:\tilde{U}\rightarrow \tilde{U'}$ such that $\varphi'\circ \lambda=\varphi$. 
      For two injections $\lambda_1,\lambda_2:(\tilde{U},\Gamma,\varphi)\rightarrow (\tilde{U'},\Gamma',\varphi')$, there exists a unique $\gamma'\in\Gamma'$ such that $\lambda_2=\gamma'\circ \lambda_1$. 
      Note that every $\gamma\in\Gamma$ defines an injection given by $\tilde{U}\ni w\mapsto \gamma\cdot w\in \tilde{U}$, since we have $\varphi(\gamma\cdot w)=\varphi(w).$

    An \textit{$n$-dimensional orbifold atlas} on $X$ is a family $\mathcal{U}=\{(\tilde{U}_i,\Gamma_i,\varphi_i)\}_{i\in I}$ of $n$-dimensional orbifold charts, such that
    \begin{enumerate} 
    \item 
    $X=\bigcup_{i\in I}\varphi_i(\tilde{U}_i)$, 
    \item for any two charts $(\tilde{U}_i,\Gamma_i, \varphi_i)$, $(\tilde{U}_j,\Gamma_j,\varphi_j)$ and a point $p\in U_i\cap U_j$, there exists an open neighborhood $U_k\subset U_i\cap U_j$ of $p$ and an orbifold chart $(\tilde{U}_k,\Gamma_k,\varphi_k)\in \mathcal{U}$  for $U_k$ such that there are injections $\lambda_{ki}:(\tilde{U}_k,\Gamma_k,\varphi_k)\rightarrow (\tilde{U}_i,\Gamma_i,\varphi_i)$ and $\lambda_{kj}:(\tilde{U}_k,\Gamma_k,\varphi_k)\rightarrow (\tilde{U}_j,\Gamma_j,\varphi_j)$.
    \end{enumerate}
        An atlas $\mathcal{U}$ is said to be a \textit{refinement} of $\mathcal{V}$ if there exists an injection on every chart of $\mathcal{U}$ into some chart of $\mathcal{V}$.
    Two orbifold atlases are said to be  \textit{equivalent} if there exists a common refinement. 
    A \textit{complex orbifold} is a paracompact Hausdorff space $X$ with an equivalence class of orbifold atlases. 
    We can define \textit{smooth orbifolds} by replacing $\mathbb{C}$ by $\mathbb{R}$ and the words ``holomorphic" by ``smooth".
    \begin{rem}
     Every orbifold atlas for $X$ is contained in a unique maximal one, and two orbifold atlases are equivalent if and only if they are contained in the same maximal one. Therefore, although we often think of $(X,\mathcal{U})$ with an orbifold atlas $\mathcal{U}$, we regard this as working with a maximal atlas.
    \end{rem}

    An \textit{orbifold vector bundle of rank $r$} on an orbifold $\mathcal{X}=(X,\mathcal{U})$ consists of a holomorphic vector bundle $E_{\tilde{U}_i}$ of rank $r$ over $\tilde{U}_i$ for each chart $(\tilde{U}_i,\Gamma_i,\varphi_i)\in\mathcal{U}$ such that for each injection $\lambda_{ji}:(\tilde{U}_i,\Gamma_i,\varphi_i)\rightarrow (\tilde{U}_j,\Gamma_j,\varphi_j)$ there exists a bundle map $E(\lambda_{ji}) :E_{\tilde{U}_i}\rightarrow E_{\tilde{U}_j}|_{\lambda_{ji}(\tilde{U}_i)}$ covering $\lambda_{ji}$, 
    and for any composition of injections $\lambda_{kj}\circ \lambda_{ji}$ we have\[E(\lambda_{kj}\circ \lambda_{ji})= E(\lambda_{kj})\circ E(\lambda_{ji}).\]
     By choosing small enough orbifold charts, we may assume that $E_{\tilde{U}_i}$ is the product vector bundle $\tilde{U}_i\times \mathbb{C}^r$. 
    Then, each bundle map $E(\lambda_{ji})$ can be written as \[E(\lambda_{ji})(w,v)=(\lambda_{ji}(w),h_{\lambda_{ji}}(w)\cdot v),\]
    where $h_{\lambda_{ji}}:\tilde{U}_i\rightarrow GL(r , \mathbb{C})$ is a holomorphic map satisfying 
    \begin{align}\label{cocycle cond}
        h_{\lambda_{kj}\circ\lambda_{ji}}(w)=h_{\lambda_{kj}}(\lambda_{ji}(w)) \circ h_{\lambda_{ji}}(w),\ \forall w\in \tilde{U}_i.
    \end{align}
    We sometimes refer to $\{h_{ji}\}$ as \textit{transition maps}.   
    Note that each $E_{\tilde{U}_i}$ can be viewed as a $\Gamma_i$-equivariant vector bundle in the following way.
     Notice that each $\gamma\in \Gamma_i$ can be viewed as an injection $\gamma: \tilde{U}_i\rightarrow \tilde{U}_i$. 
    Then, consider a map $(w,v)\mapsto (\gamma\cdot w,h_{\gamma}(w)\cdot v)$.
    We see that this defines an action of $\Gamma_i$ on $E_{\tilde{U}_i}\simeq \tilde{U}_i\times \mathbb{C}^r$ as an extension of the action of $\Gamma_i$ on $\tilde{U}_i$.
    Thus, $E_{\tilde{U}_i}$ is a $\Gamma_i$-equivariant vector bundle over $\tilde{U}_i$. 
    In a similar way, we can define smooth orbifold vector bundles and complex orbifold vector bundles. 

    Let $\{E_{\tilde{U}_i}\}$ be an orbifold vector bundle over $\mathcal{X}$. A \textit{section} of $\{E_{\tilde{U}}\}$ consists of a section $s_{\tilde{U}}$ of $E_{\tilde{U}}$ for each orbifold chart such that, for any injection $\lambda_{ji}$,
     we have
     \[s_{\tilde{U_j}}|_{\lambda_{ji}(\tilde{U}_i)}\circ \lambda_{ji}=E(\lambda_{ji})\circ s_{\tilde{U}_i}.\]

    Let $\{E_{\tilde{U}}\}$, $\{F_{\tilde{U}}\}$ be orbifold vector bundles over $\mathcal{X}=(X,\mathcal{U})$. 
    An \textit{orbifold vector bundle homomorphism} $\alpha:\{E_{\tilde{U}}\}\rightarrow \{F_{\tilde{U}}\}$ is a family of vector bundle homomorphisms $\alpha_{\tilde{U}}: E_{\tilde{U}}\rightarrow F_{\tilde{U}}$ (one for each orbifold chart)  that is compatible with the injections in the sense that for any injection $\lambda_{ji}$  we have 
    $\alpha_{\tilde{U}_i}|_{E(\lambda_{ji})(E_{\tilde{U}_i})}\circ E(\lambda_{ji})=F(\lambda_{ji})\circ \alpha_{\tilde{U}_i}$. An \textit{orbifold vector bundle isomorphism} $\alpha$ is an orbifold vector bundle homomorphism that each $\alpha_{\tilde{U}}$ is a vector bundle isomorphism.

\begin{ex}
Let $\mathcal{X}$ be a smooth orbifold. 
The \textit{tangent orbifold bundle} $T\mathcal{X}$ is defined by taking the tangent bundle $T\tilde{U}_i$ for each orbifold chart together with transition maps $\{h_{\lambda_{ji}}\}$, for any injections $\lambda_{ji}$, which are given by the Jacobian matrix of $\lambda_{ji}$. 
In a similar way, we can define the \textit{cotangent orbifold bundle} $T^{*}\mathcal{X}$.
\end{ex}

\begin{ex}
Let $\mathcal{X}$ be an $n$-dimensional complex orbifold. Let us also consider $\mathcal{X}$ as a smooth orbifold of real dimension $2n$. Then, 
for each orbifold chart, the complexified tangent bundle $T\tilde{U}_i\otimes_{\mathbb{R}}\mathbb{C}$ splits into a direct sum $T\tilde{U}_i\otimes_{\mathbb{R}}\mathbb{C}=T^{1,0}\tilde{U}_i\oplus T^{0,1}\tilde{U}_i$, where $T^{1,0}\tilde{U}_i$ and $T^{0,1}\tilde{U}_i$ are subspaces spanned by 
$\frac{\partial}{\partial w_{i1}},\ldots,\frac{\partial}{\partial w_{in}}$ and by $\frac{\partial}{\partial \overline{w_{i1}}},\ldots,\frac{\partial}{\partial \overline{w_{in}}}$, respectively. Here $w_{i1},\ldots, w_{in}$ denotes the holomorphic coordinates in $\tilde{U}_i$.
The \textit{complexified orbifold tangent bundle} $T\mathcal{X}\otimes_{\mathbb{R}}\mathbb{C}$ is a family $\{T\tilde{U}\otimes_{\mathbb{R}}\mathbb{C}\}$ whose transition maps are of the form $\Bigl(
\begin{smallmatrix}
   J(\lambda_{ji}) & 0 \\ 
   0 & \overline{J(\lambda_{ji})}
\end{smallmatrix}
\Bigr)$ 
where $J(\lambda_{ji})=\Bigl(\frac{\partial w_{jk}\circ \lambda_{ji}}{\partial w_{il}}\Bigr)_{k,l}$ is the holomorphic Jacobian matrix of $\lambda_{ji}$.
We have the decomposition
$T\mathcal{X}\otimes_{\mathbb{R}}\mathbb{C}=T^{1,0}\mathcal{X}\oplus T^{0,1}\mathcal{X}$ 
where $T^{1,0}\mathcal{X}, T^{0,1}\mathcal{X}$ are the holomorphic and the antiholomorphic orbifold tangent bundle given by $\{J(\lambda_{ji})\},\{\overline{J(\lambda_{ji})}\}$, respectively.
In a similar way, the \textit{(complexified) cotangent orbifold bundle} can be defined, and we have the decomposition $T^{*}\mathcal{X}\otimes_{\mathbb{R}}\mathbb{C}=(T^{1,0}\mathcal{X})^{*}\oplus (T^{0,1}\mathcal{X})^{*}$.
\end{ex}

A \textit{differential $k$-form} is a smooth section of $\bigwedge^{k}T^{*}\mathcal{X}$. 
Equivalently, to give a differential $k$-form $\omega$ is to give a differential $k$-form $\omega_{\tilde{U}_i}$ on each orbifold chart such that $\lambda_{ji}^{*}\omega_{\tilde{U}_j}=\omega_{\tilde{U}_i}$ for any injections $\lambda_{ji}$. 
In particular, each $\omega_{\tilde{U}_i}$ is a $\Gamma_i$-invariant.
A \textit{differential form of type (p,q)} on $\mathcal{X}$ is a smooth section of $\bigwedge^{p}(T^{1,0}\mathcal{X})^{*}\otimes \bigwedge^{q}(T^{0,1}\mathcal{X})^{*}$. 
The exterior derivative is defined as in the case of manifolds, and we have the natural decomposition $d=\partial +\bar{\partial}$.
Hence, we can define the de Rham cohomology for smooth orbifolds, and the Dolbeault cohomology for complex orbifolds in the usual way. (It is known that de Rham's theorem and Dolbeault's theorem holds in the orbifold case. For more details see \cite{satake1956generalization,baily1957imbedding}.)
We can also define a \textit{connection} $D$ on an orbifold vector bundle $\{E_{\tilde{U}_i}\}$ as a collection of connections $D_{\tilde{U}_i}$ on each $E_{\tilde{U}_i}$ that is compatible with the injections.    

A \textit{Hermitian metric} $h$ on an orbifold consists of a ($\Gamma_i$-invariant) Hermitian metric $h_{\tilde{U}_i}$ on each $\tilde{U}_i$  such that the injections are isometries, i.e., $\lambda_{ji}^{*}(h_{\tilde{U}_j}|_{\lambda_{ji}(\tilde{U}_i)})=h_{\tilde{U}_i}$. 
The imaginary part of a Hermitian metric $h$ gives rise to a real differential (1,1)-form $\omega$. 
If this $\omega$ is a closed form, we say that 
$\mathcal{X}$ is a \textit{K\"{a}hler orbifold}.

\subsection{Orbisheaves and Baily divisors}
An \textit{orbisheaf} $\mathcal{F}$ on an orbifold $\mathcal{X}$ consists of a sheaf $\mathcal{F}_{\tilde{U}}$ on $\tilde{U}$ for each
chart $(\tilde{U},\Gamma,\varphi)$ 
such that for each injection $\lambda_{ji}:(\tilde{U}_i,\Gamma_i,\varphi_i)\rightarrow (\tilde{U}_j,\Gamma_j,\varphi_j)$ there exists an isomorphism of sheaves $\mathcal{F}(\lambda_{ji}):\mathcal{F}_{\tilde{U}_i}\rightarrow \lambda_{ji}^{*}\mathcal{F}_{\tilde{U}_j}$, and these isomorphisms are required to be functorial in injections.

A \textit{map of orbisheaves} $\alpha:\mathcal{F}\rightarrow \mathcal{F}'$ is a family of sheaf maps $\alpha_{\tilde{U}}:\mathcal{F}_{\tilde{U}}\rightarrow \mathcal{F}_{\tilde{U}}^{'}$ (one for each orbifold chart), that is compatible with the injections in the sense that for each injection $\lambda_{ji}$ we have $\lambda_{ji}^{*}\alpha_{\tilde{U}_j}\circ\mathcal{F}(\lambda_{ji})=\mathcal{F}^{'}(\lambda_{ji})\circ \alpha_{\tilde{U}_i}$. 

\begin{ex}
For each chart $(\tilde{U}_i,\Gamma_i,\varphi_i)$, define $\mathcal{O}_{\tilde{U}_i}$ to be the sheaf of germs of holomorphic functions on $\tilde{U}_i$. For each injection $\lambda_{ji}:(\tilde{U}_i,\Gamma_i,\varphi_i)\rightarrow(\tilde{U}_j,\Gamma_j,\varphi_j)$, there exists an isomorphism $\mathcal{O}_{\tilde{U}_i}\rightarrow \lambda_{ji}^{*}\mathcal{O}_{\tilde{U}_j}$ defined by $\mathcal{O}_{\tilde{U}_i,w}\ni f\mapsto f\circ ({\lambda_{ji}|_{\lambda_{ji}(\tilde{U}_i)}})^{-1}\in 
\mathcal{O}_{\tilde{U}_j,\lambda_{ji}(w)}=(\lambda_{ji}^{*}\mathcal{O}_{\tilde{U}_j})_{w}$, which is functorial in the injections. Thus, we define an orbisheaf $\mathcal{O}_{\mathcal{X}}$ of an orbifold $\mathcal{X}$. We refer to it as the \textit{structure sheaf} of $\mathcal{X}$.
\end{ex}

 We can construct an orbisheaf of $\mathcal{O}_{\mathcal{X}}$-modules on $\mathcal{X}$ 
 as a collection of a sheaf of $\mathcal{O}_{\tilde{U}}$ modules on each chart $(\tilde{U},\Gamma,\varphi)$.
 An orbisheaf $\mathcal{M}$ of $\mathcal{O}_{\mathcal{X}}$-modules on $\mathcal{X}$ is said to be \textit{locally free} if for each point $p\in X$ there exists an orbifold chart $(\tilde{U},\Gamma, \varphi)$ around $p$ such that
 $\mathcal{M}_{\tilde{U}}\simeq \mathcal{O}_{\tilde{U}}^{\oplus r}$
 for some positive  integer $r$, which is called \textit{rank $r$} of $\mathcal{M}$. A locally free orbisheaf of rank one is called \textit{invertible orbisheaf}. 
 As in the case of manifolds, there is a one-to-one correspondence between isomorphism classes of holomorphic orbifold vector bundles (resp.\ holomorphic line bundles) over an orbifold $\mathcal{X}$ and isomorphism classes of locally free orbisheaves (resp.\ invertible orbisheaves) on $\mathcal{X}$. 

An \textit{orbidivisor} or \textit{Baily divisor}  on an orbifold $\mathcal{X}$ consists of a Cartier divisor $\mathcal{D}_{\tilde{U}}$ on $\tilde{U}$ for each chart $(\tilde{U},\Gamma, \varphi)$ such that 
     if $\lambda_{ji}:(\tilde{U}_i,\Gamma_i,\varphi_i)\rightarrow (\tilde{U}_j,\Gamma_j,\varphi_j)$ is an injection and $f\in (\mathcal{O}_{\tilde{U}_j}(\mathcal{D}_{\tilde{U}_j}))_{\lambda_{ji}(w)}$, then $f\circ \lambda_{ji} \in(\mathcal{O}_{\tilde{U}_i}(\mathcal{D}_{\tilde{U}_i}))_{w}$. Here, $\mathcal{O}_{\tilde{U}}(\mathcal{D}_{\tilde{U}})$ denotes the sheaf associated to $\mathcal{D}_{\tilde{U}}$.

We see that an orbidivisor $\mathcal{D}$ on $\mathcal{X}$ defines an orbifold holomorphic line bundle as follows. 
 If $\mathcal{D}_{\tilde{U}_i}$ are the divisors of the functions $f_{\tilde{U}_i}$ on $\tilde{U}_i$, we define 
 \[h_{\lambda_{ji}}=\frac{f_{\tilde{U}_j}\circ \lambda_{ji}}{f_{\tilde{U}_i}}\]
  for any injections $\lambda_{ji}$. These $h_{\lambda_{ji}}$ are nonzero holomorphic functions and satisfy the condition \eqref{cocycle cond}. Thus, we obtain an orbifold line bundle on $\mathcal{X}$ associated to $\mathcal{D}$.   
  Equivalently, to each orbidivisor, we can associate an invertible orbisheaf $\mathcal{O}_{\mathcal{X}}(\mathcal{D})$.

\subsection{Weighted projective spaces as orbifolds}   
In this subsection, we recall the weighted projective space as an orbifold. We refer to \cite{mann2005cohomologie,boyer2007sasakian}.
    
Let $Q=(q_0,\ldots,q_n)$ be an $n+1$-tuple of positive integers with $\gcd(q_0,\ldots,q_n)=1$.
Consider the weighted $\mathbb{C}^{*}$-action on $\mathbb{C}^{n+1}\backslash \{0\}$ defined by 
\begin{align}\label{action}
    (z_0,\ldots,z_n)\mapsto \lambda\cdot (z_0,\ldots,z_n):=(\lambda^{q_0}z_0,\ldots,\lambda^{q_n}z_n).
\end{align}
We denote the quotient space $(\mathbb{C}^{n+1}\backslash \{\mathbf{0}\})/\mathbb{C}^{*}$ by $\mathbf{P}(Q)$, which is classically called the weighted projective space.
However, we treat $\mathbf{P}(Q)$ as the underlying space of an orbifold $\mathbb{P}(Q)=(\mathbf{P}(Q),\mathcal{U})$, which we shall describe below, and we refer to this orbifold $\mathbb{P}(Q)$ as the weighted projective space throughout in this paper.

Set $U_i=\{[z_0:\cdots:z_n]\in\mathbf{P}(Q)\mid z_i\neq 0\}$ for $i=0,\ldots,n$, where $[z_0:\cdots:z_n]$ denotes the orbit of $(z_0,\ldots,z_n)$. 
These open sets $U_i$ cover $\mathbf{P}(Q)$. 
Let $\Gamma_i\subset\mathbb{C}^{*}$ be the subgroup of $q_i$-th roots of unity. Then, we have $U_i=\{z_i=1\}/\Gamma_i$. 
For $U_i$, we take an orbifold chart $(\tilde{U}_i,\Gamma_i,\varphi_i)$, where $\tilde{U_i}=\mathbb{C}^{n}\simeq\{z_i=1\}$ with affine coordinates $w_i=(w_{i0},\ldots,\widehat{w_{ii}},\ldots,w_{in})$ satisfying $w_{ij}^{q_i}=\frac{z_{j}^{q_i}}{z_{i}^{q_j}}$. 
The action $\Gamma_i\simeq \mathbb{Z}_{q_i}$ of $\tilde{U}_i$ is given by \[w_i\mapsto\zeta\cdot w_i=(\zeta^{q_0}w_{i0},\ldots,\widehat{w_{ii}},\ldots,\zeta^{q_n}w_{in})\] for $\zeta\in \Gamma_i$.
The map $\varphi_{i}:\tilde{U}_i\rightarrow U_i$ is given by
\[\varphi_{i}(w_i)=[w_{i0}:\ldots:1:\ldots:w_{in}],\] which induces the homeomorphism $\tilde{U}_i/\Gamma_i\xrightarrow{\simeq}U_i$.

Next, for a point $p=[z_0:\cdots:z_n]$ in overlaps, we set $I_p=\{i\in\{0,\ldots,n\}\mid z_i\neq 0\}$. Namely, $p\in \cap_{i\in I_p}U_i$. 
We consider an orbifold chart around $p$ which is induced from $(\tilde{U}_i,\Gamma_i,\varphi_i)$ for any fixed $i\in I_p$ in the following way.
Fix a point $\tilde{p}\in \varphi_{i}^{-1}(p)$, and let us consider a connected open neighborhood of $\tilde{p}$ of the form  
$D^{n}_{i}(\tilde{p},\varepsilon)=D(w_{i0}(\tilde{p}),\varepsilon)\times \cdots\times D(w_{in}( \tilde{p}),\varepsilon)\subset \tilde{U}_i$, where $D(w_{ik} (\tilde{p}),\varepsilon)$ denotes an open disk of radius $\varepsilon$ centered at $w_{ik}(\tilde{p})\in \mathbb{C}$. 
We choose a connected open neighborhood $U_{p,i}$ of $p\in U_i$ given by $U_{p,i}=\varphi_{i}(D^{n}_{i}(\tilde{p},\varepsilon))$. 
By taking smaller $\varepsilon$ if necessary, we may assume that 
$U_{p,i}\subset \cap_{i\in I_p}U_i$ and, for any $j\in I_p$, the preimage can be represented as $\varphi_{j}^{-1}(U_{p,i})=\sqcup_{\tilde{p}\in\varphi_{j}^{-1}(p)}\tilde{U}_{\tilde{p},j}$. 
Here, these $\tilde{U}_{\tilde{p},j}$ are disjoint open subsets in $\tilde{U}_j$ and $\tilde{p}\in\tilde{U}_{\tilde{p},j}$.
Then, a triple $(\tilde{U}_{\tilde{p},i},(\Gamma_i)_{\tilde{p}},\varphi_{i}|_{\tilde{U}_{\tilde{p}}})$ gives an orbifold chart for $U_{p,i}$, where $(\Gamma_i)_{\tilde{p}}$ is the stabilizer of $\tilde{p}$, which coincides with the set of $\gcd(q_{i_0},\ldots,q_{ik})$-th roots of unity when $I_p=\{i_0,\ldots,i_k\}$. 
Note that different choices of points in $\varphi_{i}^{-1}(p)$ give equivalent orbifold charts around $p$. 

For each inclusion $U_{p,i}\subset U_j$, $j\in I_p$, consider a map $\lambda_{ji}:\tilde{U}_{\tilde{p},i}\rightarrow \tilde{U}_j$ given as follows. 
For $j=i$, let $\lambda_{ii}=\text{id}_{\tilde{U}_{\tilde{p}}}$. 
For $j\in I_p\backslash \{i\}$, define $\lambda_{ji}$ by setting 
\[\lambda_{ji}(w_{i0},\ldots,w_{in})=(w_{ij}^{-\frac{q_0}{q_j}}w_{i0},\ldots,\widehat{1},\ldots,w_{ij}^{-\frac{q_0}{q_j}}w_{in}),\]
where we choose a branch of $w_{ij}^{1/q_j}$.
We see that this map satisfies the condition $\varphi_j\circ\lambda_{ji}=\varphi_{i}$, and $\lambda_{ji}$ for any $i,j\in I_p$ 
gives an injection corresponding to an inclusion $U_{p,i}\subset U_j$.  
We note that a triple $(\lambda_{ji}(\tilde{U}_{\tilde{p},i}),(\Gamma_j)_{\lambda_{ji}(\tilde{p})}
,\varphi_{j}|_{\lambda_{ji}(\tilde{U}_{\tilde{p},i})})$ also gives an orbifold chart for $U_{p,i}\subset U_j$, which are equivalent to $(\tilde{U}_{\tilde{p},i},(\Gamma_i)_{\tilde{p}},\varphi_{i}|_{\tilde{U}_{\tilde{p}}})$.
 
We see that these orbifold charts $(\tilde{U}_i,\Gamma_i,\varphi_i)$ for $i=0,\ldots,n$ and orbifold charts of the form $(\tilde{U}_{\tilde{p},i},(\Gamma_i)_{\tilde{p}},\varphi_{i}|_{\tilde{U}_{\tilde{p}}})$ together with injections constructed above give an orbifold atlas $\mathcal{U}$ on $\mathbf{P}(Q)$. 
Thus, we obtain an orbifold $\mathbb{P}(Q)=(\mathbf{P}(Q),\mathcal{U})$, which we call the weighted projective space.

\section{Preliminaries on toric orbifolds}
   
In order to fix notations, we briefly recall and collect some facts about complete toric orbifolds associated to stacky fans in the sense of Borisov-Chan-Smith \cite{borisov2005orbifold}.  By an orbifold we mean a smooth Deligne-Mumford stack whose general point has trivial stabilizer.  
This is the case when a finitely generated abelian group $N$ of a stacky fan has no torsion.
Toric orbifolds are also defined in terms of ``tours action'', which is established by Iwanari and Fantechi-Mann-Nironi, respectively. 
 We refer to \cite{borisov2005orbifold,borisov2009conjecture,fantechi2010smooth,iwanari2009category} for toric orbifolds, and \cite{cox2011toric,oda1988convex,fulton1993introduction} for toric varieties. 
\subsection{Toric orbifolds}
Let $N\simeq \mathbb{Z}^n$ be a lattice of rank $n$, and let $M=\mathrm{Hom}(N,\mathbb{Z})$ be the dual lattice with the natural paring $\langle\ ,\ \rangle:M\times N\rightarrow \mathbb{Z}$.
Let $\Sigma$ be a simplicial fan in $N_{\mathbb{R}}:=N\otimes_{\mathbb{Z}}\mathbb{R}$, i.e., every cone in $\Sigma$ is generated by linearly independent generators over $\mathbb{R}$. 
Let $\rho_0,\ldots,\rho_{m-1}$ be the one-dimensional cones of $\Sigma$, which are called rays. We denote the set of rays by $\Sigma(1)$. 
For $i=0,\ldots,m-1$, let $b_i\in N\cap \rho_i$ be a lattice point, which does not have to be the minimal lattice point.
Then we have a homomorphism of groups $\beta:\mathbb{Z}^{m}\rightarrow N$
determined by $\{b_i\}$. 
We assume that $\beta$ has finite cokernel.
The \textit{stacky fan} is the triple $\boldsymbol{\Sigma}=(N,\Sigma,\beta)$. We call the $b_i$'s \textit{stacky vectors} \footnote{We take the term ``stacky vectors'' by \cite{cho2014holomorphic}, for example.}.

\begin{rem}
    If $\Sigma$ is a complete fan, i.e., its support $\cup_{\sigma\in \Sigma}\sigma$ is the whole space $N_{\mathbb{R}}$, then the $b_i$'s generate $N_{\mathbb{Q}}$, which implies that the assumption can be satisfied. 
\end{rem}

A stacky fan $\boldsymbol{\Sigma}=(N,\Sigma,\beta)$ defines a toric orbifold as follows. 
Let $\beta^{*}:M\rightarrow (\mathbb{Z}^{m})^{*}$  be the dual map of $\beta$, i.e., $\beta^{*}(m)=(\langle m,b_0\rangle, \ldots,\langle m,b_{m-1}\rangle)$, where we abbreviate $\mathrm{Hom}(-,\mathbb{Z})$ by $(-)^{*}$. 
We see that $\beta^{*}$ is injective by the assumption.
We denote the cokernel of $\beta^{*}$ by $DG(\beta)$. 
We then have the following exact sequence
\begin{align}\label{divisor sequence}
    0\rightarrow M \xrightarrow{\beta^{*}}(\mathbb{Z}^m)^{*}\rightarrow DG(\beta)\rightarrow 0,
\end{align}
called the \textit{divisor sequence}.
Applying $\mathrm{Hom}(- ,\mathbb{C}^*)$ to (\ref{divisor sequence}) gives,
\[1\rightarrow \mathrm{Hom}(\mathrm{DG(\beta)},\mathbb{C}^*)\rightarrow (\mathbb{C}^*)^m\rightarrow T_N:=\mathrm{Hom}(M,\mathbb{C}^{*})\rightarrow 1,\]
which remains to be right exact since $\mathbb{C}^{*}$ is divisible.
Let $Z:=\mathbb{C}^m\backslash \mathbf{V}(J_{\Sigma})$  be the open subset of $\mathbb{C}^m$ with the coordinate ring $\mathbb{C}[z_0,\ldots,z_{m-1}]$, defined by the ideal  $J_{\Sigma}:=\langle \prod_{\rho_i\not \subset \sigma}z_i\mid \sigma\in\Sigma\rangle$. 
We set 
\[G:=\mathrm{Hom}(\mathrm{DG}(\beta),\mathbb{C}^*).\]
Then, the algebraic group $G\simeq\mathrm{Ker}((\mathbb{C}^{*})^m\rightarrow T_N)$
determines the induced action of $(\mathbb{C}^{*})^{m}$ on an open subset $Z\subset \mathbb{C}^m$ as $(t_0,\ldots,t_{m-1})\cdot(z_{0},\ldots,z_{m-1})=(t_{0}z_{0},\ldots,t_{m-1}z_{m-1}).$
We thus obtain a quotient stack
\[\mathcal{X}_{\boldsymbol{\Sigma}}:=[Z/G],\]
which becomes a smooth Deligne-Mumford stack with generically trivial stabilizer, and its coarse moduli space is a toric variety $X_{\Sigma}$ (See Proposition 3.7 in \cite{borisov2005orbifold}).  
A toric orbifold has the action of a DM torus $\mathcal{T}=[(\mathbb{C}^{*})^m/G]$ with an open dense orbit isomorphic to $\mathcal{T}$. In fact, this DM torus is an ordinary torus $T=(\mathbb{C}^{*})^m/G\simeq \text{Spec}(\mathbb{C}[M])$. For more details, see \cite{fantechi2010smooth}.

\vspace{\baselineskip}
\noindent \textbf{Open substacks.}
Let $\sigma$ be a cone in $\Sigma$ of maximal dimension $n$.
Viewing a cone $\sigma\in\Sigma$ as the fan consisting of the cone $\sigma$ and all its faces, we can identify $\sigma$ with an open substack of $\mathcal{X}_{\mathbf{\Sigma}}$ (See \cite{borisov2005orbifold} Proposition 4.3).
Let $\beta_{\sigma}:\mathbb{Z}^{n}\rightarrow N$ be the map determined by the set $\{b_i\mid \rho_i\subset  \sigma\}$. Then the stacky fan $\boldsymbol{\sigma}=(N,\sigma,\beta_{\sigma})$ yields an open substack $\mathcal{X}_{\boldsymbol{\sigma}}=[Z_\sigma/G_{\sigma}]$ of $\mathcal{X}_{\boldsymbol{\Sigma}}$. 
We write $N_{\sigma}=\text{Im}\beta_{\sigma}$, which is the sublattice spanned by $\{b_i\mid \rho_i\subset \sigma\}$  and has finite index in $N$, i.e., $N/N_{\sigma}$ is a finite group. 
Then, we see that $G_{\sigma}\simeq N/N_{\sigma}$ and  $\mathcal{X}_{\boldsymbol{\sigma}}$ is expressed as 
\[[\mathbb{C}^n/(N/N_{\sigma})].\]

Let us describe this more explicitly. 
Denote by $M_{\sigma}$ the dual lattice of $N_{\sigma}$.
Let us relabel the generators the $b_i$ of $N_{\sigma}$ as $b_{(i)1},\ldots,b_{(i)n}$, and let  
$b^{*}_{(i)1},\ldots,b^{*}_{(i)n}$ be the dual basis of $M_{\sigma}$, i.e., $\langle b^{*}_{(i)k},b_{(i)l}\rangle=\delta_{k,l}$.
Then, we see that  $Z_{\sigma}=\mathbb{C}^n$ can be written as 
\[\text{Spec}(\mathbb{C}[M_{\sigma}\cap\sigma^{\vee}])=\text{Spec}(\mathbb{C}[\chi^{{b^{*}_{(i)1}}},\ldots,\chi^{{b^{*}_{(i)n}}}])=\mathbb{C}^{n}.\]

The action of $N/N_{\sigma}$ on $\text{Spec}(\mathbb{C}[M_{\sigma}\cap \sigma^{\vee}])$ is given as in toric varieties. See e.g., subsection 1.3 in \cite{cox2011toric} or subsection 1.5 in \cite{oda1988convex}. 
Let $[n]\in N/N_{\sigma}$, and let $w\in\text{Spec}(\mathbb{C}[M_{\sigma}\cap \sigma^{\vee}])$, where we consider a point $w$ as a semigroup homomorphism $\mathbb{C}[M_{\sigma}\cap \sigma^{\vee}]\rightarrow \mathbb{C}$ sending $ m\mapsto \chi^{m}(w)$. Then, $[n]\cdot w$ is defined by a semigroup homomorphism
\[m\mapsto e^{-2\pi\sqrt{-1}\langle m,n\rangle}\chi^{m}(w), \quad m\in M_{\sigma}\cap\sigma^{\vee}.\] 
Here, we denote by $\langle\ ,\ \rangle$ a $\mathbb{Z}$-bilinear map $M_{\sigma}\times N\rightarrow \mathbb{Q}$ which is a common extension of $\langle\ ,\ \rangle :M\times N\rightarrow \mathbb{Z}$ and $\langle\ ,\ \rangle :M_{\sigma}\times N_{\sigma}\rightarrow \mathbb{Z}$. 
Note that the composition $M_{\sigma}/M\times N/N_{\sigma}\rightarrow \mathbb{Q}/\mathbb{Z}\rightarrow \mathbb{C}^{*}$ given by $([m_{\sigma}],[n])\mapsto e^{2\pi\sqrt{-1}\langle m_{\sigma},n\rangle}$ is well-defined.
Furthermore, we see that the moduli space $U_{\sigma}=\text{Spec}(\mathbb{C}[M\cap\sigma^{\vee}])\subset X_{\Sigma}$ of an open substack $\mathcal{X}_{\boldsymbol{\sigma}}\subset \mathcal{X}_{\mathbf{\Sigma}}$ is expressed as 
\[U_{\sigma}=\text{Spec}(\mathbb{C}[M_{\sigma}\cap \sigma^{\vee}]^{N/N_{\sigma}})\simeq \text{Spec}(\mathbb{C}[M_{\sigma}\cap\sigma^{\vee}])/(N/N_{\sigma}).\]

In terms of orbifolds, for $U_{\sigma}$ we have an orbifold chart of the form
\[(\tilde{U}_{\sigma}:=\text{Spec}(\mathbb{C}[M_{\sigma}\cap\sigma^{\vee}])=\mathbb{C}^{n},N/N_{\sigma},\varphi_{\sigma}),\]
where $\varphi_{\sigma}:\tilde{U}_{\sigma}\rightarrow U_{\sigma}$ is a natural projection which factors through $\tilde{U}_{\sigma}/(N/N_{\sigma})$.
If we denote the coordinates of $\tilde{U}_{\sigma}=\mathbb{C}^n$ by $(w_{(i)1},\ldots,w_{(i)n})$, then the action of $N/N_{\sigma}$ is given by 
\[w_{(i)k}\mapsto e^{-2\pi\sqrt{-1}\langle b^{*}_{(i)k},n\rangle}w_{(i)k}, \quad k=1,\ldots,n.\]

\subsection{Weighted projective spaces as toric orbifolds}
Let us describe a toric structure of the weighted projective space $\mathbb{P}(q_0,\ldots,q_n)$ with $\gcd(q_0,\ldots,q_n)=1$. 
Set $l:=\mathrm{lcm}(q_0,\ldots,q_n)$. 
Let $N$ be a lattice of rank $n$ which is generated by vectors $-\frac{l}{q_0}\sum_{i=0}^{n}e_i, \frac{l}{q_1}e_1,\ldots,\frac{l}{q_n}e_n$, where the $e_i$ are the standard basis of $\mathbb{Z}^{n}$. 
Note that these $n+1$ vectors are linearly independent since $q_0(-\frac{l}{q_0}\sum_{i=0}^{n}e_i)+\sum_{i=0}^{n}q_i(\frac{l}{q_i}e_i)=0$. 
Let $\Sigma$ be a fan in $N_{\mathbb{R}}$ whose cones are generated by proper subsets of $\{e_0:=-\sum_{i=1}^{n}e_i,e_1,\ldots,e_n\}$. This $\Sigma$ is a complete and simplicial fan with $(n+1)$-rays $\rho_i=\mathbb{R}_{\geq0}e_i$ for $i=0,\ldots,n$. Then, set 
\[b_0:=-\frac{l}{q_0}\sum_{i=0}^{n}e_i,\quad b_i:=\frac{l}{q_i}e_i,i=0,\ldots,n.\]
Let these $b_i,i=0,\ldots,n$, be stacky vectors. 
Notice that we have
\begin{align}\sum_{i=0}^nq_ib_i=0,\end{align}
which implies that we have the following exact sequence
\[0\rightarrow\mathbb{Z}\xrightarrow{\!^{t}(q_0\ \ldots \ q_n)}\mathbb{Z}^{n+1}\xrightarrow{(b_0\ \ldots \ b_n)}N\rightarrow 0.\]
Since $N$ is free, we obtain the divisor sequence as
\begin{align}\label{divisor sequence for WPS}
0\rightarrow M\xrightarrow{(\langle \ ,b_0\rangle\ \ldots \ \langle \ , b_n\rangle)}\mathbb{Z}^{n+1}\xrightarrow{(q_0\ \ldots \ q_n)}\mathbb{Z}\rightarrow 0. 
\end{align}
We then see that $Z=\mathbb{C}^{n+1}\backslash\{0\}$ and the action $G=\mathbb{C}^{*}$ on $Z$ is given by
\[(z_0,\ldots,z_n)\mapsto (\lambda^{q_0}z_0,\ldots,\lambda^{q_n}z_n), \quad \lambda\in\mathbb{C}^{*}.\]
Thus, we obtain the weighted projective space $\mathbb{P}(q_0,\ldots,q_n)=[(\mathbb{C}^{n+1}\backslash\{0\})/\mathbb{C}^{*}]$ as a toric orbifold, 
and its coarse moduli space is a toric variety $\mathbf{P}(q_0, \ldots , q_n)$, which is also referred to as the weighted projective space, e.g., in \cite{fulton1993introduction}, subsection $2.2$. 

\begin{ex}[$\mathbb{P}(3,2)$]
    The lattice is $N=\mathbb{Z}e_1$. The fan $\Sigma$ in $N_{\mathbb{R}}=\mathbb{R}e_1$ consists of $\sigma_0=[0,\infty)$, $\sigma_1=[0,-\infty)$ and $\sigma_{01}=\{0\}$, and the stacky vectors are $b_0=-2e_1,b_1=3e_1$. 
     Figure \ref{p(3,2)} shows this stacky fan of the weighted projective line $\mathbb{P}(3,2)$.
\end{ex}

\begin{ex}[$\mathbb{P}(1,1,2)$]
    The lattice $N$ is generated by $-2(e_1+e_2),2e_1,e_2$. In particular, we have  $N=2\mathbb{Z}e_1+\mathbb{Z}e_2$,
    and we identify $N$ with $\mathbb{Z}^2$ via $2e_1\mapsto e_1,e_2\mapsto e_2$.
    In $\mathbb{R}^2$, the fan $\Sigma$ consists of $\sigma_0,\sigma_1,\sigma_2$ shown in Figure \ref{p(1,1,2)} together with all its faces. The stacky vectors $b_0,b_1,b_2$ are as in Figure \ref{p(1,1,2)}. 
\end{ex}
    \begin{figure}[htbp]
    \centering
     \begin{minipage}[c]{0.48\linewidth}
\begin{tikzpicture}[scale=0.7]
    \draw(-3.5,0)--(3.5,0);
       \foreach \x in {-3,...,3}
    \fill[black](\x,0)circle(0.07);
    \draw(0,0)node[above]{0}; 
    \draw [very thick, -{Latex[length=2.5mm]}] (0,0) -- (3.9,0);
    \draw [very thick,-{Latex[length=2.5mm]}] (0,0) -- (-3.9,0);
    \draw(-2,0)node[above ]{$b_0=-2e_1$};
    \draw(3,0)node[above]{$b_1=3e_1$};
    \draw(1.5,0)node[below]{$\sigma_0$};
    \draw(-1.5,0)node[below]{$\sigma_1$};
    \filldraw[fill=white,  thick, draw=black](3,0)circle[radius=0.1];
    \filldraw[fill=white, thick, draw=black](-2,0)circle[radius=0.1];
\end{tikzpicture}
\caption{The stacky fan of $\mathbb{P}(3,2)$.}\label{p(3,2)}

     \end{minipage}
    \begin{minipage}[c]{0.48\linewidth}
        \centering
        \begin{tikzpicture}[scale=1.3]         
    \foreach \x in {-2,...,2}
    \foreach \y in {-2,...,2}\fill[black](\x/2,\y/2)circle(0.04); 
        \draw [very thick,-{Latex[length=2.5mm]}] (0,0) -- (0,3/2);
        \draw[very thick,-{Latex[length=2.5mm]}](0,0)--(3/2,0);
        \draw[very thick,-{Latex[length=2.5mm]}](0,0)--(-1.5/2,-3/2);
        \draw(1/2,0)node[below]{$b_1$};
        \draw(0,1/2)node[left]{$b_2$};
        \draw(-0.5,-1)node[above left]{$b_0$};
        \draw(0,0)node[left]{$0$};
    \filldraw[fill=white, thick, draw=black](1/2,0)circle[radius=0.06];
     \filldraw[fill=white, thick, draw=black](0,1/2)circle[radius=0.06];
      \filldraw[fill=white, thick, draw=black](-0.5,-1)circle[radius=0.06];
         \draw(2.5/2,2.5/2)node{$\sigma_0$};
         \draw(2.5/2,-2.5/2)node{$\sigma_2$};
         \draw(-3/2,0)node{$\sigma_1$};
        \end{tikzpicture}
        \caption{The \mbox{stacky fan of $\mathbb{P}(1,1,2)$.}} \label{p(1,1,2)}
        \end{minipage}
        
\end{figure}

    In terms of open substacks, we see that an orbifold charts for $U_i=U_{\sigma_i}=\{[z_0:
    \cdots:z_n]\in \mathbf{P}(q_0,\ldots,q_n)\mid z_i\neq0\}$
    defined in section 2.3 coincides with that  given by open substacks. 
    In particular, we see that 
    \[N/N_{\sigma_i}\simeq \mathbb{Z}/q_i\mathbb{Z}.\]
    In fact, let set 
    $(b_{(i)1},\ldots,b_{(i)n})=(b_0,\ldots,\widehat{b_i},\ldots,b_n)$,
    for each maximal dimensional cone $\sigma_i$ for $i=0,\ldots,n$.
    Note that any lattice point $\sum_{k=0}^{n}n_kb_k\in N=\mathbb{Z}^{n+1}/\mathbb{Z}(q_0,\ldots,q_n)$ can be written as $\sum_{k\neq i}(n_k-\frac{n_iq_k}{q_i})b_k$, and we have $\sum_{k=0}^{n}n_kb_k \equiv -\frac{n_i}{q_i}\sum_{k\neq i}q_kb_k$ modulo  $N_{\sigma_i}=\sum_{k\neq i}\mathbb{Z}b_k$.  
    Consider a homomorphism of abelian groups $\mathbb{Z}\rightarrow N/N_{\sigma_i}$ defined by $n\mapsto -\frac{n}{q_i}\sum_{k\neq i}q_kb_k$. 
    Since $\gcd(q_i,\gcd(q_0,\ldots,\hat{q_i},\ldots,q_n))=1$, the kernel of this map coincides with $q_i\mathbb{Z}$.
    Therefore, $N/N_{\sigma_i}\simeq\mathbb{Z}/q_i\mathbb{Z}$. 
    We see that the action of $\mathbb{Z}/q_i\mathbb{Z}\simeq N/N_{\sigma_i}$ is given as follows. For $[n]\in \mathbb{Z}/q_i\mathbb{Z}$, 
    \begin{multline*}
        [n]\cdot(w_{(i)1},\ldots,w_{(i)n})=\\
        (e^{\frac{2nq_{0}}{q_i}\pi\sqrt{-1}}w_{(i)1},\ldots,e^{\frac{2nq_{i-1}}{q_i}\pi\sqrt{-1}}w_{(i)i},e^{\frac{2nq_{i+1}}{q_i}\pi\sqrt{-1}}w_{(i)i+1},\ldots,e^{\frac{2nq_{n}}{q_i}\pi\sqrt{-1}}w_{(i)n}),
    \end{multline*}
    which coincides with the action of $\Gamma_i=\{\text{the $q_i$-th roots of unity}\}$.

\subsection{Invertible sheaves on toric orbifolds}
    Let $\mathcal{X}_{\boldsymbol{\Sigma}}=[Z/G]$ be a toric orbifold associated to a stacky fan $\boldsymbol{\Sigma}=(N, \Sigma, \beta)$. 
    It is known that a coherent sheaf on a Deligne-Mumford stack $[Z/G]$ is a $G$-equivariant coherent sheaf on $Z$ (see \cite{vistoli1989intersection}).
    Since $Z$ is obtained from $\mathbb{C}^{\Sigma(1)}$ by removing a subspace of codimension at least two, an invertible sheaf on $[Z/G]$ is determined by the structure sheaf $\mathcal{O}_{Z}$ and a character of $G$, $\chi \in\mathrm{Hom}(G,\mathbb{C}^{*})$. 
    By the divisor sequence, note that we have 
    \[\mathrm{Pic}(\mathcal{X}_{\boldsymbol{\Sigma}})\simeq \mathrm{Hom}(G,\mathbb{C}^{*})\simeq DG(\beta)\simeq\mathbb{Z}^{\Sigma(1)}/\beta^{*}(M).\] 
    Such an invertible sheaf can be identified with a sheaf of $\chi$-equivariant regular sections of the trivial line bundle over $Z$ with the $G$-linearization determined by $\chi$.
    In particular, we have the following explicit description of it (see \cite{borisov2009conjecture}, Definition 3.1).
    For $\chi\in \mathrm{Hom}(G,\mathbb{C}^{*})$, if we consider $\chi$ as a character of $(\mathbb{C}^{*})^{\Sigma(1)}$, there exists $(a_0,\ldots,a_{m-1})\in\mathbb{Z}^{\Sigma(1)}$ such that $\chi(t_0,\ldots,t_{m-1})=\prod_{i=0}^{m-1}t_{i}^{a_i}, t=(t_0,\ldots,t_{m-1})\in G\subset (\mathbb{C}^{*})^{\Sigma(1)}$.
    Then the $G$-linearization $G\times Z\times \mathbb{C}\longrightarrow Z\times \mathbb{C}$ of the trivial line bundle over $Z\times\mathbb{C}\rightarrow Z$ is given by 
        \[(t,z,v)\longmapsto (t\cdot z,\prod_{i=0}^{m-1}t_{i}^{a_i}v).\]
    We denote the corresponding invertible sheaf by $\mathcal{O}(\sum_{i=0}^{m-1}a_i\mathcal{D}_i)$. 
    We note that $t\in(\mathbb{C}^{*})^{\Sigma(1)}$ lies in $G\simeq\mathrm{Ker}((\mathbb{C}^{*})^m\rightarrow T_N)$ if and only if $\prod_{i=0}^{m-1}t_{i}^{\langle m,b_i\rangle}=1$ for all $m\in M$. 
    So, if an element $(a_{1}^{'},\ldots,a_{m-1}^{'})\in \mathbb{Z}^{\Sigma(1)}$ satisfies
    \[(a_{1}^{'},\ldots,a_{m-1}^{'})=(a_0,\ldots,a_{m-1})+(\langle m,b_0\rangle,
    \ldots,\langle m,b_{m-1}\rangle)\]
    for some $m\in M$, then $(a_{1}^{'},\ldots,a_{m-1}^{'})$ gives the same $G$-linearization. 
    Namely, $\mathcal{O}(\sum_{i=0}^{m-1}a_i\mathcal{D}_i)=\mathcal{O}(\sum_{i=0}^{m-1}a_{i}^{'}\mathcal{D}_i)$ in $\text{Pic}(\mathcal{X}_{\boldsymbol{\Sigma}})$.

\subsection{Orbifold line bundles over the weighted projective space.}
    Let us return to the weighted projective spaces.
   By the divisor sequence, we have $\mathrm{Pic}(\mathbb{P}(q_0,\ldots,q_n))\simeq \mathbb{Z}$.
   For each $a\in\mathbb{Z}$, there exists $(a_0,\ldots,a_n)\in\mathbb{Z}^{n+1}$ such that  $\sum_{i=0}^{n}q_ia_i=a$, 
   and we write $\mathcal{O}(a)=\mathcal{O}(\sum_{i=0}^{n}a_iD_i)\in \mathrm{Pic}(\mathbb{P}(q_0,\ldots,q_n))$.
   We note that it has the following expression of the corresponding $G$-linearization.
    \[((z_0,\ldots,z_n),v)\mapsto ((\lambda^{q_0}z_0,\ldots,\lambda^{q_n}z_n),\lambda^{a}v),\quad \lambda\in\mathbb{C}^{*}=G. \]

   Let us return to orbifold points of view for the later discussions. We use the same notations in subsection 2.3.
   We want to construct the corresponding orbifold line bundle, which we also denote by $\mathcal{O}(a)$.
   For a toric divisor $[(\{z_i=0\}\cap Z)/G]$, which we identify with $\mathcal{D}_i$, consider an orbidivisor given by $\{(\tilde{U}_k, f_k)\}$, where 
   \[f_k=w_{ki}=\frac{z_i}{{{z_{k}^{\frac{q_k}{q_i}}}}}\] 
   for $k\neq i$, and $f_i=1$.
   Then, for $\mathcal{O}(a)=\mathcal{O}(\sum_{i=0}^{n}a_i\mathcal{D}_i)$ with $\sum_{i=0}^nq_ia_i=a$, consider an orbidivisor $\{(\tilde{U}_k,\prod_{i\neq k}w_{ki}^{a_i})\}$, and we obtain the corresponding orbifold line bundle whose action of $\Gamma_k=\{q_k\text{-th roots of unity}\}$ for each orbifold chart $\tilde{U}_k$ is given by 
   \[((w_{i0},\ldots,\hat{w}_{ii},\ldots, w_{in}),v)\mapsto ((\gamma^{q_0}w_{i0},\ldots,\hat{w_{ii}} \ldots,\gamma^{q_n}w_{in}),\gamma^{a}v), \ \gamma\in\Gamma_i\]
   and transition maps are given by
   \[h_{\lambda_{ij}}=\left(\frac{z_{j}^{1/{q_j}}}{z_{i}^{1/{q_i}}}\right)^{a}.\]

\section{The SYZ torus fibration set-up}
In this section, we discuss an extension of the SYZ torus fibration set-up to toric orbifolds as an analogue of that in \cite{leung2000special,leung2005mirror,chan2009holomorphic,futaki2021homological}. 
More precisely, we discuss the SYZ construction for a DM torus of a toric orbifold, which is isomorphic to the ordinary torus of the underlying toric variety, and we denote this torus by $\check{Y}$. We construct its mirror manifold $Y$ in subsection 4.1. We demonstrate it in the case of the weighted projective spaces in subsection 4.2. 
In subsection 4.3, we discuss the SYZ transformation, where we allow Lagrangian submanifolds of $Y=T^{*}{N}_{\mathbb{R}}/2\pi M$ to be shifted in the fiber direction by elements in $2\pi M_{\mathbb{Q}}$ (modulo $2\pi M$) and assign holomorphic line bundles on $\check{Y}$ equipped with connections. We demonstrate it in the case of the weighted projective spaces in subsection 4.4. 

\subsection{Dual torus fibrations for a toric orbifold}
 For the case of smooth toric varieties, the SYZ construction is applied for the open dense torus orbit $\check{Y}=(\mathbb{C}^{*})^n$, which is also the complement of toric divisors $=X_{\Sigma}\backslash \cup_{\rho\in\Sigma(1)}D_{\rho}$, in order to obtain its mirror manifold $Y$. Analogously, we discuss the SYZ construction for the DM torus of a toric orbifold.

Let $\mathcal{X}_{\mathbf{\Sigma}}$ be an $n$-dimensional complete toric orbifold associated to a stacky fan $\mathbf{\Sigma}=(N,\Sigma,\beta)$.
A toric orbifold $\mathcal{X}_{\mathbf{\Sigma}}=[Z/G]$ has the action of the torus $(\mathbb{C}^{*})^{\Sigma(1)}/G$, which is isomorphic to the ordinary torus $\text{Spec}(\mathbb{C}[M])\simeq(\mathbb{C}^{*})^n$ of $X_{\Sigma}$. 
 We set 
    \[\check{Y}=(\mathbb{C}^{*})^{\Sigma(1)}/G.\]
 We shall consider a torus fibration whose total space is $\check{Y}$ and construct its dual torus fibration $Y$ over the same base, where the base space is an affine manifold. 
        
We first recall that an affine manifold is a smooth manifold equipped with an affine open covering whose coordinate transformations are all affine maps. While it is well known that the cotangent bundle of a smooth manifold is a symplectic manifold with the standard symplectic form, the tangent bundle of an affine manifold becomes a complex manifold.
For more details, see, e.g., \cite{leung2000special,leung2005mirror}.
Now, let us consider $N_{\mathbb{R}}$ as an affine manifold with an open covering $\{(N_{\sigma_i})_{\mathbb{R}}\}_{\sigma_i\in\Sigma(n)}$, where the local coordinates $\check{x}_{(i)1},\ldots,\check{x}_{(i)n}$ of $(N_{\sigma_i})_{\mathbb{R}}\simeq\mathbb{R}^n$ are taken with respect to the base $b_{(i)1},\ldots,b_{(i)n}$ for each $\sigma_i$, and $\Sigma(n)$ denotes the set of the maximal dimensional cones in $\Sigma$.
If we denote the fiber coordinates of $T^{*}N_{\mathbb{R}}|_{(N_{\sigma_{i}})_{\mathbb{R}}}$ and $TN_{\mathbb{R}}|_{(N_{\sigma_{i}})_{\mathbb{R}}}$ by $(y^{(i)1},\ldots,y^{(i)n})$ and $(\check{y}_{(i)1},\ldots,\check{y}_{(i)n})$, respectively, then the corresponding structures on them are locally given by $\sum_{k=1}^nd\check{x}_{(i)k}\wedge dy^{(i)k}$ and $\check{x}_{(i)k}+\sqrt{-1}\check{y}_{(i)k}, k=1,\ldots,n$, respectively.  
Then we obtain dual torus fibrations \[T^{*}N_{\mathbb{R}}/2\pi M\rightarrow N_{\mathbb{R}},\quad TN_{\mathbb{R}}/2\pi N\rightarrow N_{\mathbb{R}},\]
and they become symplectic and complex manifold, respectively.
    
Then, we can consider $\check{Y}$ as $TN_{\mathbb{R}}/2\pi N$ in the following way. In order to describe the local trivialization of it, let us describe this torus $\check{Y}\simeq\text{Spec}(\mathbb{C}[M])=\bigcap_{\sigma_i\in\Sigma(n)}U_{\sigma_i}$ in terms of orbifold charts for $U_{\sigma_i}$. 
Recall that for each maximal dimensional cone $\sigma_{i}\in\Sigma(n)$, we have an orbifold chart for $U_{\sigma_{i}}$, determined by the corresponding open substack $\mathcal{X}_{\boldsymbol{\sigma_i}}\subset \mathcal{X}_{\boldsymbol{\Sigma}}$, of the form 
    \[(\tilde{U}_{\sigma_i}=\mathrm{Spec}(\mathbb{C}[M_{\sigma_i}\cap \sigma_{i}^{\vee}])=\mathbb{C}^n,\ N/N_{\sigma_i},\ \varphi_{\sigma_i})\] 
    with coordinates $w_{(i)1},\ldots,w_{(i)n}$ of $\tilde{U}_{\sigma_i}=\mathbb{C}^n$.
We note that the ordinary torus $\text{Spec}(\mathbb{C}[M])\subset U_{\sigma_i}\simeq\tilde{U}_{\sigma_i}/(N/N_{\sigma_i})$ can be written as 
    \[\text{Spec}(\mathbb{C}[M])=\text{Spec}(\mathbb{C}[M_{\sigma_i}]^{N/N_{\sigma_i}})\simeq\text{Spec}(\mathbb{C}[M_{\sigma_i}])/(N/N_{\sigma_i}),\]
and the right hand side gives the expression of $\check{Y}$ with the inclusion $\check{Y}\hookrightarrow U_{\sigma_i}$.
Also note that the torus $\text{Spec}(\mathbb{C}[M_{\sigma_i}])=(\mathbb{C}^{*})^{n}$ of the upper space $\tilde{U}_{\sigma_i}$ can be identified with $T(N_{\sigma_i})_{\mathbb{R}}/2\pi N_{\sigma_i}\simeq(N_{\sigma_i})_{\mathbb{R}}\times2\pi(N_{\sigma_i})_{\mathbb{R}}/2\pi N_{\sigma_i}$ by setting  \[w_{(i)k}=e^{\check{x}_{(i)k}+\sqrt{-1}\check{y}_{(i)k}},k=1,\ldots,n.\]
Therefore, $\check{Y}$ is locally written as $(T(N_{\sigma_i})_{\mathbb{R}}/2\pi N_{\sigma_i})/(N/N_{\sigma_i})=T(N_{\sigma_i})_{\mathbb{R}}/2\pi N$, and we obtain a torus fibration $\check{p}:\check{Y}\rightarrow N_{\mathbb{R}}$ which is locally expressed as 
    \[\check{Y}|_{(N_{\sigma_i})_{\mathbb{R}}}\simeq(N_{\sigma_i})_{\mathbb{R}}\times2\pi(N_{\sigma_i})_{\mathbb{R}}/2\pi N\longrightarrow (N_{\sigma_i})_{\mathbb{R}},\]
    \[(\check{x}_{(i)1},\ldots,\check{x}_{(i)n},\check{y}_{(i)1},\ldots,\check{y}_{(i)n})\longmapsto (\check{x}_{(i)1},\ldots,\check{x}_{(i)n}),\]
where $(\check{y}_{(i)1},\ldots,\check{y}_{(i)n})$ denotes the fiber coordinates of $\check{Y}|_{(N_{\sigma_i})_{\mathbb{R}}}$ by abuse of notation.
    
Correspondingly, we set $Y=T^{*}N_{\mathbb{R}}/2\pi M$ and consider a torus fibration $p:Y\rightarrow N_{\mathbb{R}}$ which is locally expressed as  
    \[Y|_{(N_{\sigma_i})_{\mathbb{R}}}\simeq(N_{\sigma_i})_{\mathbb{R}}\times2\pi(M_{\sigma_i})_{\mathbb{R}}/2\pi M\longrightarrow (N_{\sigma_i})_{\mathbb{R}},\]
    \[(\check{x}_{(i)1},\ldots,\check{x}_{(i)n},y^{(i)1},\ldots,y^{(i)n})\longmapsto (\check{x}_{(i)1},\ldots,\check{x}_{(i)n}).\]
With the descending structures from $TN_{\mathbb{R}}$ and $T^{*}N_{\mathbb{R}}$, $\check{Y}$ and $Y$ are complex and symplectic manifolds, respectively. 
    
\begin{rem}
        If a maximal dimensional cone $\sigma_i\subset N_{\mathbb{R}}$ is smooth, i.e., the ray generators of the edges of $\sigma_i$ extend to a $\mathbb{Z}$-basis of $N$, then the torus fibers of $Y|_{(N_{\sigma_i})_{\mathbb{R}}}$ and $\check{Y}|_{(N_{\sigma_i})_{\mathbb{R}}}$ are $\mathbb{R}^n/2\pi\mathbb{Z}^n$. 
        However, if $\sigma_i$ is not smooth, which is the case that the action of $N/N_{\sigma_i}$ is not trivial, then the period of the torus fibers are affected by the difference between the sublattice $N_{\sigma_i}$ (resp.$\ M_{\sigma_i}$) and the lattice $N$ (resp.$\ M$). We revisit this situation in the case of the weighted projective line $\mathbb{P}(3,2)$ in the next subsection. (See Example \ref{the ex of torus fibrations}.) 
\end{rem}
    
    For later discussions, let us equip $\mathcal{X}_{\boldsymbol{\Sigma}}$ with a K\"{a}hler structure.\footnote{For details of symplectic (K\"{a}hler) toric orbifolds, see e.g., \cite{lerman1997hamiltonian}.}
    Let $\mu:\mathcal{X}_{\boldsymbol{\Sigma}}\rightarrow 
     P\subset M_{\mathbb{R}}$ be the moment map, where $P$ is the moment polytope.
     Then, we see that the restriction of it to $\check{Y}$ 
     gives us another expression of this torus fibrations $p:Y\rightarrow N_{\mathbb{R}},\ \check{p}:\check{Y}\rightarrow N_{\mathbb{R}}$ as that over the same base $B:=\text{Int}(P)$,
     \[\pi:Y\rightarrow B,\quad \check{\pi}:\check{Y}\rightarrow B,\]
     as follows. 
    Given a K\"{a}hler structure on $\mathcal{X}_{\boldsymbol{\Sigma}}$, denote a K\"{a}hler form on $\tilde{U}_{\sigma_i}$ by 
    $\omega_{{\sigma_i}}$, and denote the corresponding moment map by
    $\mu_{\sigma_i}:\tilde{U}_{\sigma_i}\rightarrow  P_{\sigma_i}\subset (M_{\sigma_i})_{\mathbb{R}}\simeq\mathbb{R}^n$, where 
    $(M_{\sigma_i})_{\mathbb{R}}=M_{\mathbb{R}}\xrightarrow{\simeq}\mathbb{R}^n$ is given by   
    $m\mapsto (\langle m,b_{(i)1}\rangle,\ldots,\langle  m,b_{(i)n}\rangle)+p_i,\ p_i\in\mathbb{R}^n$.
    The restriction of $\omega_{\sigma_i}$ to $(\mathbb{C}^{*})^n\subset\tilde{U}_{\sigma_i}$ can be expressed as 
    \[ 2\sqrt{-1}\partial\overline{\partial}\check{\phi}_{\sigma_i}=\sum_{k,l}\frac{\partial^2\check{\phi}_{\sigma_i}}{\partial \check{x}_{(i)k}\partial \check{x}_{(i)l}}d\check{x}_{(i)k}\wedge d\check{y}_{(i)l}\]
    for some smooth function $\check{\phi}_{\sigma_i}:(N_{\sigma_i})_{\mathbb{R}}\rightarrow \mathbb{R}$. The restricted moment map 
    $\mu_{\sigma_i}|_{(\mathbb{C}^{*})^n}:(\mathbb{C}^{*})^n\rightarrow \text{Int}P_{\sigma_i}=:B_{\sigma_i}$ is  expressed as 
    \[ (\check{x}_{(i)1},\ldots,\check{x}_{(i)n},\check{y}_{(i)1},\ldots,\check{y}_{(i)n})\longmapsto \left(\frac{\partial\check{\phi}_{\sigma_i}}{\partial \check{x}_{(i)1}},\ldots,\frac{\partial\check{\phi}_{\sigma_i}}{\partial \check{x}_{(i)n}}\right). \]
    This map is $N/N_{\sigma_i}$-invariant, and descends to a map $\check{Y}\rightarrow  B$. We denote it by $\check{\pi}$.
    Notice that we have a diffeomorphism $\Phi:N_{\mathbb{R}}\rightarrow B$ which is locally given by 
    \[(\check{x}_{(i)1},\ldots,\check{x}_{(i)n})\longmapsto \left(\frac{\partial\check{\phi}_{\sigma_i}}{\partial \check{x}_{(i)1}},\ldots,\frac{\partial\check{\phi}_{\sigma_i}}{\partial \check{x}_{(i)n}}\right),\]
     and we have $\check{\pi}=\Phi\circ \check{p}$.
    Set $\pi=\Phi\circ p$.  We thus obtain torus fibrations $\pi:Y\rightarrow B,\ \check{\pi}:\check{Y}\rightarrow B$ with the same base $B$. 
    In later discussions, we identify  $p,\check{p}$ with $\pi,\check{\pi}$, respectively, via an identification $ N_{\mathbb{R}}\underset{\Phi}{\simeq}B$.

    We note that we can equip $B$ with a Hessian metric $g$ which is induced by a given K\"{a}hler metric, which we use to define the category $Mo(P)$ of weighted Morse homotopy on $P=\overline{B}$ in subsection 5.2. The K\"{a}hler metric of $\mathcal{X}_{\boldsymbol{\Sigma}}$ is locally expressed as $\sum_{k,l=1}^{n}\frac{\partial^2\check{\phi}_{\sigma_i}}{\partial \check{x}_{(i)k}\partial \check{x}_{(i)l}}(d\check{x}_{(i)k}\otimes d\check{x}_{(i)l}+d\check{y}_{(i)k}\otimes d\check{y}_{(i)l})$ when restricted to $\check{Y}$, and it induces a Hessian metric $\check{g}$ on $N_{\mathbb{R}}$ which is locally given by $\sum_{k,l=1}^{n}\frac{\partial^2\check{\phi}_{\sigma_i}}{\partial \check{x}_{(i)k}\partial \check{x}_{(i)l}}d\check{x}_{(i)k}\otimes d\check{x}_{(i)l}$.  
    Let us set 
    \[\check{g}_{(i)}^{kl}=\frac{\partial^2\check{\phi}_{\sigma_i}}{\partial \check{x}_{(i)k}\partial\check{x}_{(i)l}}.\]
    Then, we equip $B$ with a metric $g$ which is locally given by 
    \[\sum_{k,l=1}^ng^{(i)}_{kl}dx^{(i)k}\otimes dx^{(i)l}\]
    where $(x^{(i)1},\ldots,x^{(i)n})$ denotes the dual coordinates of $B_{\sigma_i}$ and $(g^{(i)}_{kl})$ is the inverse matrix of $(\check{g}_{(i)}^{kl})$.

    \begin{rem}
        By choosing a Hessian metric $g$ on the base space $B$, $TB$ is equipped with a symplectic structure and $T^{*}B$ is equipped with a complex structure via the isomorphism 
        $TB\xrightarrow{\simeq} T^{*}B$  induced by $g$. Moreover, both $TB$ and $T^{*}B$ turn to be K\"{a}hler manifolds. 
        Then, $Y\simeq TB/2\pi M$ and $\check{Y}\simeq T^{*}B/2\pi N$ become K\"{a}hler manifolds with the induced structures. 
    \end{rem}
    
    \subsection{Dual torus fibrations for weighted projective space}    
    Let us return to the weighted projective space. 
    By applying the construction above to the torus  $\check{Y}=(\mathbb{C}^{*})^{n+1}/\mathbb{C}^{*}\simeq (\mathbb{C}^{*})^n$ of $\mathcal{X}_{\boldsymbol{\Sigma}}=\mathbb{P}(Q)$ (for the toric structure of $\mathbb{P}(Q)$, see subsection 3.2), we obtain its dual torus fibration 
    $Y$:
    \[Y=T^{*}N_{\mathbb{R}}/2\pi M\rightarrow N_{\mathbb{R}}\simeq B,\quad \check{Y}=TN_{\mathbb{R}}/2\pi N\rightarrow N_{\mathbb{R}}\simeq B,\]
    where we will fix a K\"{a}hler structure on $\check{Y}$ later in this subsection.
    Let us give an example of these torus fibrations in the case of $\mathbb{P}(3,2)$.
    
    \begin{ex}[$\mathbb{P}(3,2)$] 
        \label{the ex of torus fibrations} 
    Recall that the sublattice $N_{\sigma_0}\subset N$ is $\mathbb{Z}b_{(0)1}=3\mathbb{Z}e$ and $N=\frac{1}{3}\mathbb{Z}b_{(0)1}$. The dual lattice $M_{\sigma_0}\supset M$ is $\mathbb{Z}b^{*}_{(0)1}=\frac{1}{3}\mathbb{Z}e^{*}$ and $M=3\mathbb{Z}b^{*}_{(0)1}$. 
    Then, for $ p\in (N_{\sigma_0})_{\mathbb{R}}$, the torus fibers are given as follows:
    \[Y_p\simeq \{p\}\times 2\pi (M_{\sigma_0})_{\mathbb{R}}/2\pi M\simeq  \mathbb{R}/(2\pi\cdot3\mathbb{Z}),\]  
    \[\left.\check{Y}_p\simeq \{p\}\times 2\pi (N_{\sigma_0})_{\mathbb{R}}/2\pi N\simeq  \mathbb{R}\middle/\left(2\pi\cdot\frac{1}{3}\mathbb{Z}\right).\right.\]
    \end{ex}

    \begin{rem}
        We can also consider $\check{Y}\subset\mathbb{P}(Q)=(\mathbf{P}(Q),\mathcal{U})$ as a complex manifold whose structure is given by the restriction of the orbifold atlas $\mathcal{U}$, since $\check{Y}$ has no orbifold singular points, i.e., every point in $\check{Y}$ has the trivial stabilizer. 
        We can identify this structure of a manifold $\check{Y}$ and the one coming from $\check{Y}=T{N}_{\mathbb{R}}/2\pi N$. 
    \end{rem}

    We equip $\check{Y}\subset\mathbb{P}(Q)=(\mathbf{P}(Q),\mathcal{U})$ with a K\"{a}hler form locally given by
  \[-2\sqrt{-1}\partial\bar{\partial}\log(1+\sum_{j\neq i}(w_{ij}\overline{w_{ij}})^{\frac{q_0\cdots q_n}{q_j}}),\]
    which is of the form when embedded into $\tilde{U}_i=\tilde{U}_{\sigma_i}$\footnote{This K\"{a}hler form actually degenerates at the origin of $\tilde{U}_i$ when we naturally extend it to $\tilde{U}_i$. However, it is enough to consider a K\"{a}hler form on $\check{Y}$. If we consider the corresponding labeled polytope given by stacky vectors $b_i=c_iv_i$, where $c_i$ are positive integers and $v_i$ are the first lattice points in the rays, we can take a K\"{a}hler structure over the whole $\mathbb{P}(Q)$. For more details, see, e.g., \cite{lerman1997hamiltonian,cho2014holomorphic}.}.
    The restriction to $(\mathbb{C}^{*})^n\subset \tilde{U}_{\sigma_i}$ can be written as 
    \[ \sum_{k,l\neq i}\frac{\partial^2\check{\phi_i}}{\partial\check{x}_{ik}\partial\check{x}_{il}}d\check{x}_{ik}\wedge d\check{y}_{il},\quad \check{\phi_i}:=\mathrm{log}(1+\sum_{j\neq i}e^{2\frac{q_0\cdots q_n}{q_j}\check{x}_{ij}}),\]
    where we set $w_{ik}=e^{\check{x}_{ik}+\sqrt{-1}\check{y}_{ik}}, k=0,\ldots,\widehat{i},\ldots,n$. 
    Then, $\check{\pi}:\check{Y}\rightarrow B$ is given by $(\check{x}_{(i)},\check{y}_{(i)})\mapsto (\frac{\partial \check{\phi_i}}{\partial \check{x}_{(i)}})$. 
    Note that the coordinates $\check{x}_{(i)}=(\check{x}_{i0},\ldots,\widehat{\check{x}_{ii}},\ldots,\check{x}_{in})$ of $(N_{\sigma_i})_{\mathbb{R}}$ and the dual coordinates $x^{(i)}=(x^{i0},\ldots,\widehat{x^{ii}},\ldots,x^{in})$ of $B_{\sigma_i}\subset (M_{\sigma_i})_{\mathbb{R}}$ are related by 
    \[x^{il}=\frac{\partial\check{\phi_i}}{\partial\check{x}_{il}}=\frac{2q_0\cdots q_n e^{2\frac{q_0\cdots q_n}{q_l}\check{x}_{il}}}{q_l(1+\sum_{j\neq i}e^{2\frac{q_0\cdots q_n}{q_j}\check{x}_{ij}})}, \quad l=0,\ldots,\widehat{i},\ldots,n.\]
    Here, $B_{\sigma_i}$ is explicitly written as 
    \[B_{\sigma_i}=\{(x^{i0},\ldots,\widehat{x^{ii}},\ldots,x^{in})\mid x^{il}>0,l\neq i, \sum_{l\neq i}q_lx^{il}<2q_0\cdots q_n\}.\]
    For example, Figure \ref{polytope of P(3,2)} and Figure \ref{polytope of P(1,1,2)} show the polytopes $P_{\sigma_i}=\overline{B_{\sigma_i}}$ of the weighted projective line $\mathbb{P}(3,2)$ and the weighted projective plane $\mathbb{P}(1,1,2)$. 

    \begin{figure}[htbp]
        \centering
        \begin{minipage}[c]{0.45\linewidth}
            \centering
               \begin{tikzpicture}[scale=0.6]
                   \draw[->,>=stealth](-0.3,0)--(7,0)node[right]{$x^{(0)1}$};
                   \draw(0,0)node[above=2mm]{$0$};
                   \draw(6,0)node[above=2mm]{$6$};
                   \draw[very thick](0,0)--(6,0);
                   \draw(0,0)node{$[$};
                   \draw(6,0)node{$]$};
               \end{tikzpicture}
               \caption{The polytope $P_{\sigma_0}=[0,6]$ of $\mathbb{P}(3,2)$.}
               \label{polytope of P(3,2)}
            \end{minipage}
        \quad
        \begin{minipage}[c]{0.45\linewidth}
        \centering
            \begin{tikzpicture}[scale=0.7]
                 \draw[->,>=stealth,semithick](-0.3,0)--(4.5,0)node[above]{$x^{(0)1}$};
            \draw[->,>=stealth,semithick](0,-0.5)--(0,2.5)node[left]{${x^{(0)2}}$};
                \draw (0,0)--(0,2)--(4,0)--cycle;
                \draw(4,0) node[below]{4}; 
                \draw(0,2) node[left]{2}; 
                \draw(0,0)node[below left]{$0$};
                \draw[very thick](0,0)--(0,2);
                \draw[very thick](0,0)--(4,0);
                \draw[very thick](4,0)--(0,2);
            \end{tikzpicture}
            \caption{The polytope $P_{\sigma_0}$ of $\mathbb{P}(1,1,2)$}
            \label{polytope of P(1,1,2)}
        \end{minipage}
    \end{figure}
   
    \subsection{Lagrangian submanifolds of $Y$ and holomorphic line bundles on $\check{Y}$}
    Leung-Yau-Zaslow \cite{leung2000special} and Leung \cite{leung2005mirror} discussed a version of Fourier-Mukai transformation, called the SYZ transformation, which gives a correspondence from Lagrangian sections of $Y\rightarrow B$ to holomorphic line bundles with $U(1)$-connections on $\check{Y}$. 
    In \cite{futaki2021homological} based on this, and they however start from pairs of a holomorphic line bundle over a smooth compact toric manifold $X_{\Sigma}$ and a connection on it, and reconstruct Lagrangian sections such that the SYZ transformed pairs are isomorphic to the restrictions of given pairs on $\check{Y}=X_{\Sigma}\backslash \cup_{\rho \in \Sigma(1)}D_{\rho}$. 
    We extend this discussion to the case of toric orbifolds. 
    For convenience of description, we here treat dual torus fibrations $Y,\check{Y}$ as that over the base $N_{\mathbb{R}}(\simeq B)$, although 
    when we consider the category $Mo(P)$ of weighted Morse homotopy on the polytope $P$ (see section 5.2), 
    we treat dual torus fibrations $Y,\check{Y}$ as that over $B=\text{Int}P$. 

    We fix an affine open covering $\{(N_{\sigma_i})_{\mathbb{R}}\}_{\sigma_i\in\Sigma(n)}$ of $N_{\mathbb{R}}$.
    Let $\underline{s}=\{\underline{s}^{(i)}\}_{\sigma_i\in \Sigma(n)}$ be a Lagrangian section of $Y=T^{*}N_{\mathbb{R}}/2\pi M$. 
    Let $[K]\in \mathbb{Z}^{\Sigma(1)}/\beta^{*}(M)$.
    Now, we shall assign a holomorphic line bundle equipped with a connection over $\check{Y}$ to a given Lagrangian section $\underline{s}$ with $[K]$ in the following way. 
    Here, a section $\underline{s}$ of $Y$ is \textit{Lagrangian section}
    if and only if a lift $s=\{s^{(i)}\}_{\sigma_i\in\Sigma(n)}$ of $\underline{s}$ to the covering space $T^{*}N_{\mathbb{R}}$ can be locally expressed as $\sum_{k=1}^{n}s^{(i)k}d\check{x}_{(i) k}=df^{(i)}$ for some smooth function $f^{(i)}$. 
     We refer to such a local function $f^{(i)}$ as a \textit{potential} of the lift $s$.
    
    Take a lift $s=\{s^{(i)}\}_{\sigma_i\in\Sigma(n)}$ of $\underline{s}$.
    For a local expression $s^{(i)}=(s^{(i)1},\ldots,s^{(i)n}):(N_{\sigma_i})_{\mathbb{R}}\rightarrow  2\pi (M_{\sigma_i})_{\mathbb{R}}=\mathbb{R}^n$, let us write $K^{(i)}=(k^{(i)1},\ldots,k^{(i)n})\in M_{\sigma_i}=\mathbb{Z}^n$, where $k^{(i)l}$ are given by elements in $K=(k_0,\ldots,k_{m-1})\in\mathbb{Z}^{\Sigma(1)}$ satisfying 
    \[ \sum_{l=1}^n\Big\langle \sum_{j=1}^nk^{(i)j}b^{*}_{(i)j}, b_{(i)l} \Big\rangle b_{(i)l}=\sum_{\rho_l\subset \sigma_i}\Big\langle \sum_{j=0}^{m-1} k_{j}e^{*}_{j},e_l\Big\rangle b_{l}. \]
    Let us set a $U(1)$-connection on the trivial line bundle over  $\check{Y}|_{(N_{\sigma_i})_{\mathbb{R}}}$ as 
    \begin{align}\label{connection-comesfromLagrangian}
        D_{(i)}:=d-\frac{\sqrt{-1}}{2\pi}\sum_{l=1}^{n}\left(s^{(i)l}(\check{x})+2\pi k^{(i)l}\right)d\check{y}_{(i)l}.
    \end{align}
    We note that if a cone $\sigma_i\in\Sigma(n)$ is smooth, then $s^{(i)}+2\pi K^{(i)}$ is a lift of $\underline{s^{(i)}}$ of $Y|_{(N_{\sigma_i})_{\mathbb{R}}}=T^{*}(N_{\sigma_i})_{\mathbb{R}}/2\pi M$ since the shift of $2\pi K^{(i)}$, which lies in $2\pi M_{\sigma_i}=2\pi M$, corresponds to the zero section of $Y|_{(N_{\sigma_i})_{\mathbb{R}}}$. 
    Hence, the above form is the usual one.
    However, if $\sigma_i\in \Sigma(n)$ is not smooth, the shift of $2\pi K^{(i)}\in2\pi M_{\sigma_i}(\supsetneq 2\pi M)$ may not correspond to the zero section and $s^{(i)}+2\pi K^{(i)}$ may determine a distinct section from $\underline{s^{(i)}}$.
    Also note that regardless of whether a cone $\sigma_i$ is smooth or not, if $K$ comes from an element $m\in M\simeq \beta^{*}(M)\subset \mathbb{Z}^{\Sigma(1)}$, i.e., $K=(\langle m,b_0\rangle,\ldots,\langle m,b_{m-1}\rangle)$, then the shift of $2\pi K^{(i)}=2\pi (\langle m,b_{(i)1}\rangle, \ldots,\langle m,b_{(i)n} \rangle)\in 2\pi M$ can be ignored as a descending section. 

    We see that $\{D_{(i)}\}$ globally defines a connection, which we denote by $D$, on a line bundle $V$ over $\check{Y}$ whose transition functions are given by 
    \[h_{ij}=e^{\sqrt{-1}(K^{(i)} ({}^{t}\!\varphi_{ji}^{-1})-K^{(j)})\cdot {}^{t}\check{y}_{(j)}},\]
    where $\{\varphi_{ij}\}$ 
    denote the transition functions of $Y=T^{*}N_{\mathbb{R}}/M$ (or $\{{}^{t}\!\varphi_{ji}^{-1}\}$ denote the transition functions of $\check{Y}=TN_{\mathbb{R}}/N$) and ${}^{t}\check{y}_{(i)}$ denotes ${}^{t}(\check{y}_{(i)1},\ldots,\check{y}_{(i)n})$.
    Note that another lift $s+2\pi m=\{s^{(i)}+2\pi (\langle m,b_{(i)1}\rangle,\ldots,\langle m,b_{(i)n}\rangle)\},m\in M$ of $\underline{s}$ or another $K'$ such that $K'-K=(\langle m,b_0\rangle,\ldots,\langle m,b_{m-1}\rangle)\in \beta^{*}(M)\subset \mathbb{Z}^{\Sigma(1)}$ for some $m\in M$, gives the same line bundle on $\check{Y}$. Therefore, this assignment is well-defined. 

    The curvature is locally expressed as 
    \[D^2_{(i)}=\frac{\sqrt{-1}}{2\pi}\sum_{k,l=1}^{n}\frac{\partial s^{(i)k}}{\partial \check{x}_{(i)l}}d\check{x}_{(i)l}\wedge d\check{y}_{(i)k},\]
    and the $(0,2)$-part of $D^2$ vanishes since $\underline{s}$ is Lagrangian. 
    Then, this connection $D$ defines a holomorphic structure on $V$.
    Conversely, if a connection $D$ is locally of the form of \eqref{connection-comesfromLagrangian} and the $D^{(0,2)}$ vanishes, then $\underline{s}$ is Lagrangian section. 

    \begin{rem}
        We do not further discuss what class of Lagrangian sections of $Y$ correspond to holomorphic line bundles on $\check{Y}$ that can be extended over the whole $\mathcal{X}_{\boldsymbol{\Sigma}}$. In \cite{chan2009holomorphic}, Chan discussed the SYZ transformation for any smooth projective toric manifold $X_{\Sigma}$ in this point of view. There, he introduced the growth condition for Lagrangian sections, and established a bijective correspondence of the SYZ transformation between $T_{N}(=N_{\mathbb{R}}/N)$-invariant hermitian metrics on holomorphic line bundles $\mathcal{L}_{[a]}$ and Lagrangian sections satisfying growth conditions $(*_{[a]})$ for $[a]\in \mathbb{Z}^{\Sigma(1)}/\iota(M)\simeq \mathrm{Pic}(X_{\Sigma})$, where the isomorphism comes from the divisor sequence. This result gives a bijective correspondence between isomorphic classes of holomorphic line bundles over $X_{\Sigma}$ and equivalence classes of Lagrangian sections of its mirror $(Y,W)$.  
        We expect an analogous discussion holds for projective toric orbifold (\cite{growthconditionfortoricorbifolds}).
    \end{rem}

    \subsection{Lagrangian submanifolds of $Y$ corresponding to $\mathcal{O}(a)$} 

    Let us demonstrate the discussion in the previous subsection in the case of $\mathbb{P}(q_0,\ldots,q_n)$. We begin by equipping $\mathcal{O}(a)$, for $a\in \mathbb{Z}$, with a connection $d+A_a$ in order to reconstruct corresponding Lagrangian sections $\underline{s_a}$. 
    Consider a connection $d+A_a$ which is locally given by connection one-forms such as 
    \begin{align}\label{connection form}
            (A_{a})_{(i)}=-a\frac{\sum_{k\neq i}1/q_k\cdot(w_{ik}\overline{w_{ik}})^{\frac{q_0\cdots q_n}{q_k}}\cdot w_{ik}^{-1}dw_{ik}}{1+\sum_{k\neq i}(w_{ik}\overline{w_{ik}})^{\frac{q_0\cdots q_n}{q_k}}}
    \end{align}
    over an orbifold chart $\tilde{U}_i$, which is $\mathbb{Z}/q_i\mathbb{Z}$-invariant. 
    We can see that these are compatible with the injections over the intersections and determine a connection globally.
    Next, we restrict $\mathcal{O}(a)$ with $d+A_a$ on $\check{Y}=TN_{\mathbb{R}}/N$ and twist it by 
    \begin{align} \label{twisting}
        (\Psi_{a})_{(i)}:=(1+\sum_{k\neq i}e^{2\frac{q_0\cdots q_n}{q_k}\check{x}_{ik}})^\frac{a}{2q_0\cdots q_n}.
    \end{align}
    We then obtain a line bundle over $\check{Y}$, whose transition functions are given by $(\Psi_{a})_{(i)}^{-1}f_{ij}(\Psi_a)_{(j)}=e^{\sqrt{-1}\frac{a}{q_j}\check{y}_{ij}}$, 
    equipped with a connection of the form
    \begin{align}\label{twisted connection}
            (\Psi_{a})_{(i)}^{-1}(d+(A_{a})_{(i)})(\Psi_a)_{(i)} 
            &=d-\sqrt{-1}\sum_{k\neq i}
            \frac{a\cdot e^{2\frac{q_0\cdots q_n}{q_k}\check{x}_{ik}}}{q_k\left(1+\sum_{l\neq i}e^{2\frac{q_0\cdots q_n}{q_k}\check{x}_{il}}\right)}d\check{y}_{ik}. 
    \end{align}
    
    By comparing this with the expression of \eqref{connection-comesfromLagrangian}, we take $K_a\in\mathbb{Z}^{n+1}$ so that $\{s_i\}$ defines a section of $T^{*}N_{\mathbb{R}}\rightarrow N_{\mathbb{R}}\simeq B$.
    We see that such $K_a=(a_0,\ldots,a_n)$ must satisfy 
    \[\sum_{l=0}^{n}q_la_l=a,\]  
    and we obtain the corresponding section whose lifts $s_{a;K_a}$ are expressed as 
    \begin{align}\label{laglangian-section-lift}
        \begin{pmatrix}
            {y}^{(i)1} \\  \\ \vdots \\  \\
            {y}^{(i)n}
        \end{pmatrix}
        =
        \begin{pmatrix}
            s_{a;K_a}^{(i)1}(x) \\ \\ \vdots \\  \\  s_{a;K_a}^{(i)n}(x)
        \end{pmatrix}
        =2\pi\frac{a}{2q_0\ldots q_n} 
        \begin{pmatrix}
            x^{(i)1}\\  \\ \vdots \\   \\ x^{(i)n} 
        \end{pmatrix}
        -2\pi \begin{pmatrix}
                a_0 \\ \vdots \\ a_{i-1} \\ a_{i+1} \\ \vdots \\  a_n
            \end{pmatrix}.
    \end{align}
    We note that we have $M\simeq \beta^{*}(M)=\{(m_0,\ldots,m_n)\in\mathbb{Z}^{n+1}\mid \sum_{l=0}q_lm_l=0\}$ by the divisor sequence \eqref{divisor sequence for WPS}, and
    the corresponding section $\underline{s_{a;K_a}}$ is well-defined. We will often drop the index and simply write $\underline{s_a}$. 

    \begin{rem}
        We observe that our Lagrangian sections $\underline{s_a}=\underline{s_{a;K_a}}$ for $a\in \mathbb{Z}$ come from $[K_a]\in \mathbb{Z}^{n+1}/\beta^{*}(M)$ such that 
        $K_a=(k_0,\ldots,k_n)\in\mathbb{Z}^{n+1}$ corresponds to $\mathcal{O}(a)=\mathcal{O}(\sum_{j=0}^nk_j\mathcal{D}_j)\in\text{Pic}(\mathcal{X}_{\boldsymbol{\Sigma}})\simeq\mathbb{Z}^{n+1}/\beta^{*}(M)$.
    \end{rem}
    
    We see that these $\underline{s_a},a\in \mathbb{Z}$, are indeed Lagrangian sections, as the corresponding local functions that satisfy $df_{a;K_a}^{(i)}=\sum_{j\neq i }s_{a;K_a}^{ij}d\check{x}_{ik}$ are given by 
    \begin{align}\label{potential}
        f_{a;K_a}^{(i)}&=2\pi\frac{a}{2q_0\cdots q_n}\mathrm{log}(1+\sum_{k\neq i}e^{2\frac{q_0\cdots q_n}{q_k}}\check{x}_{ik})-2\pi\sum_{k\neq i}a_{k}\check{x}_{ik}
    +const \nonumber \\
        &=-2\pi\frac{q_ia_i}{2q_0\cdots q_n}\log(2q_0\cdots q_n-\sum_{k\neq i}q_kx^{ik}) \\
        & \quad \quad -2\pi\sum_{k\neq i}\frac{q_ka_k}{2q_0\cdots q_n}\log(q_kx^{ik}) \nonumber \\
        & \quad \quad +2\pi\frac{a}{2q_0\cdots q_n}\log(2q_0\cdots q_n)+ const.\nonumber
    \end{align}
  
    \begin{ex}
        Let us describe Lagrangian sections $\underline{s_a}$ for $a\in\mathbb{Z}$ in the case of the weighted projective line $\mathbb{P}(3,2)$. 
        Recall that the dual torus fibration is locally expressed as $Y|_{B_{\sigma_0}}\simeq B_{\sigma_0}\times 2\pi (M_{\sigma_0})_{\mathbb{R}}/M$, where $B_{\sigma_0}=\{0<x^{(0)1}<6\}$, and the period of the torus fiber is $6\pi$ because of $M=3M_{\sigma_0}$ (cf. Example \ref{the ex of torus fibrations}). 
        A lift $s^{(i)}_{a;K_a}$ of $\underline{s_a}$ is locally of the form
        \[s_{a;K_a}^{(0)1}=2\pi\frac{a}{2}\cdot\frac{x^{(0)1}}{6}-2\pi a_1\]
        for some $K_a=(a_0,a_1)\in\mathbb{Z}^2$ satisfying $3a_0+2a_1=a$. 
        We note that $\lim_{x^{(0)1}\rightarrow +0}s_{a;K_a}^{(0)1}(x^{(0)1})=-2\pi a_1\in 2\pi \mathbb{Z}(=2\pi M_{\sigma_0})$ is uniquely determined modulo $6\pi$.
        Figure \ref{graphofLagrangiansections} shows the graphs of lifts $s_{a;K_a}$ of Lagrangian sections $\underline{s_a}$ for $0\leq a\leq 4$. 
        For example, lifts of $\underline{s_0}$, which corresponds to the structure sheaf $\mathcal{O}$,  appear in $y^{(0)1}=6k\pi$ for $k\in \mathbb{Z}$, and so on.  
    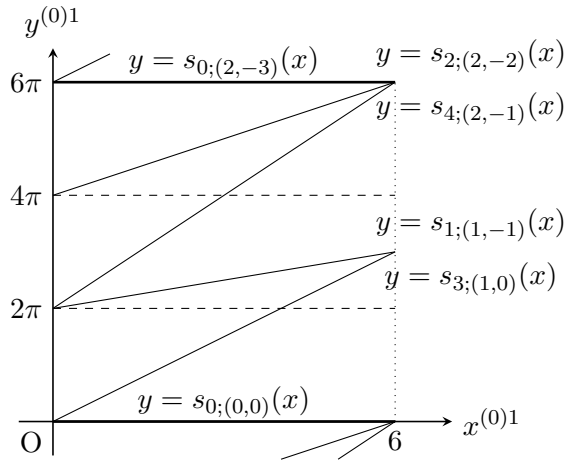
\begin{figure}[htbp]
        \centering
        \begin{tikzpicture}[scale=1.5]
            \draw[->,>=stealth,semithick](-0.3,0)--(3.5,0)node[right]{$ x^{(0)1}$};
            \draw[->,>=stealth,semithick](0,-0.3)--(0,3.3)node[above]{$y^{(0)1}$};
            \draw[dashed](0,1)--(3.0,1);
            \draw[dashed](0,2)--(3.0,2);
            \draw[line width=1pt](0,0)--(3,0);
            \draw[line width=1pt](0,3)--(3,3);
            \draw(0,1)node[left]{$2\pi$};                \draw(0,2)node[left]{$4\pi$};            \draw(0,3)node[left]{$6\pi$};
            \draw(3,0)node[below]{$6$};
            \draw(0,0)node[below left]{O};
    
            \draw[,domain=0:3]plot(\x,\x/6+1)node[right=10mm,above]{$y=s_{1;(1,-1)}(x)$};
            \draw[,domain=0:3]plot(\x,2*\x/6+2)node[right=10mm,above]{$y=s_{2;(2,-2)}(x)$};
                    \draw[,domain=2:3]plot(\x,2*\x/6-1);
            \draw[,domain=0:3]plot(\x,3*\x/6+0)node[right=10mm,below]{$y=s_{3;(1,0)}(x)$};
                    \draw[,domain=0:0.5]plot(\x,3*\x/6+3);
            \draw[,domain=0:3]plot(\x,4*\x/6+1)node[right=10mm,below]{$y=s_{4;(2,-1)}(x)$};
                    \draw[,domain=2.5:3]plot(\x,4*\x/6-2);
            \draw(1.5,0)node[above=-1mm]{$y=s_{0;(0,0)}(x)$};
            \draw(1.5,3)node[above=-1mm]{$y=s_{0;(2,-3)}(x)$};
            \draw[dotted](3,0)--(3,3);           
        \end{tikzpicture}
         \caption{The graphs of lifts of Lagrangian sections $\underline{s_{a}}$ corresponding to $\mathcal{O}(a)$ on $\mathbb{P}(3,2)$}
         \label{graphofLagrangiansections}
    \end{figure}
    \end{ex}

\section{Categories for homological mirror symmetry set-up}
In this section, we recall categories on both sides in the homological mirror symmetry in the sense of \cite{futaki2021homological} with a little modification to toric orbifolds $\mathcal{X}_{\boldsymbol{\Sigma}}$. We continue to use notations in the previous section 4, and let $\pi:Y\rightarrow B$ and $\check{\pi}:\check{Y}\rightarrow B$ be the SYZ dual torus fibration.

\subsection{DG categories of line bundles $DG(\mathcal{X})$ and $\mathcal{V}=\mathcal{V}(\check{Y})$}
    In this subsection, we mainly follow subsection 4.1 in \cite{futaki2021homological} or subsection 2.2 in \cite{futaki2022homological}. 
We define a DG-category $\mathcal{V}=\mathcal{V}(\check{Y})$ of holomorphic line bundles over $\check{Y}$ as follows. 
The objects are holomorphic line bundles $V$ with $U(1)$-connections $D$ on $\check{Y}$ associated to lifts $s$ of Lagrangian sections as we defined in section 4. 
For any two objects $s_a=(V_a,D_a)$, $s_b=(V_b,D_b)$, the space of morphisms is defined by
\[\mathcal{V}(s_a,s_b):=\Gamma(V_a,V_b)\underset{C^{\infty}(\check{Y})}{\otimes}\Omega^{0,*}(\check{Y}), \]
where $\Omega^{0,*}(\check{Y})$ denotes the space of  antiholomorphic differential forms and $\Gamma(V_a,V_b)$ denotes the space of homomorphisms from $V_a$ to $V_b$. 
This space is $\mathbb{Z}$-graded vector spaces, where the grading is given by the degree of the antiholomorphic differential forms. We denote the degree $r$ part by $\mathcal{V}^r(s_a,s_b)$. 
Next, we define the differential $d_{ab}$ on $\mathcal{V}(s_a,s_b)$ as follows. Decompose $D_a$ into its holomorphic part and antiholomorphic part $D_a=D_{a}^{(1,0)}+D_{a}^{(0,1)}$. 
We define a linear map $d_{ab}:\mathcal{V}^r(s_a,s_b)\rightarrow \mathcal{V}^{r+1}(s_a,s_b)$ as
\[d_{ab}(\psi):=2(D_{b}^{(0,1)}\psi-(-1)^r\psi D_{a}^{(0,1)})\]
for $\psi \in \mathcal{V}^{r}(s_a,s_b)$. We see  that $d_{ab}^2=0$, since $(D_{a}^{(0,1)})^2=0$.
The product structure $m:\mathcal{V}\otimes \mathcal{V}(s_b,s_c)\rightarrow \mathcal{V}(s_a,s_c)$ is given by combining the composition of bundle homomorphisms and the wedge product: For $\psi_{ab}\in \mathcal{V}^{r_{ab}}(s_a,s_b)$ and $\psi_{bc}\in \mathcal{V}^{r_{bc}}(s_b,s_c)$, 
\[m(\psi_a,\psi_b):=(-1)^{r_{ab}r_{bc}}\psi_{bc}\wedge\psi_{ab}(=\psi_{ab}\wedge\psi_{bc}).\]
Note that $d_{ab}$ satisfies Leibniz rule with respect to $m$. Thus, we see that $\mathcal{V}$ forms a DG category.

Next, we define the DG category $DG(\mathcal{X}_{\boldsymbol{\Sigma}})$ of holomorphic orbifold line bundles on $\mathcal{X}_{\boldsymbol{\Sigma}}$.
For an orbifold line bundle $V$ on $\mathcal{X}_{\boldsymbol{\Sigma}}$, we take a holomorphic connection $D$ whose restriction to $\check{Y}$ is isomorphic to a line bundle on $\check{Y}$ with a connection of the form 
\[d-\frac{\sqrt{-1}}{2\pi}\sum_{l=1}^{n}\left(s^{l}(\check{x})+2\pi k^{l}\right)d\check{y}_{l},\]
where these $k^{l}$ come from $K\in\mathbb{Z}^{\Sigma(1)}$ such that $s$ defines a lift of section of $Y$.
The objects of $DG(\mathcal{X}_{\boldsymbol{\Sigma}})$ are such pairs $s:=(V,D)$. The space $DG(\mathcal{X}_{\boldsymbol{\Sigma}})(s_a,s_b)$ of morphisms is defined as a graded vector space whose graded piece $DG^r(\mathcal{X}_{\boldsymbol{\Sigma}})(s_a,s_b), r\in \mathbb{Z}$ is given by 
\[DG^{r}(\mathcal{X}_{\boldsymbol{\Sigma}})(s_a,s_b):=\Gamma(V_a,V_b)\underset{C^{\infty}(\mathcal{X}_{\boldsymbol{\Sigma}})}{\otimes}\Omega^{0,r}(\mathcal{X}_{\boldsymbol{\Sigma}}).\]
Here, $\Gamma(V_a,V_b)$ is the space of smooth orbifold bundle morphism form $V_a$ to $V_b$. The composition of morphisms is defined in a similar way as in $\mathcal{V}(\check{Y})$.
The differential $d_{ab}:DG^r(\mathcal{X}_{\boldsymbol{\Sigma}})(s_a,s_b)\rightarrow  DG^{r+1}(\mathcal{X}_{\boldsymbol{\Sigma}})(s_s,s_b)$ is defined by 
\[d_{ab}(\tilde{\psi}):=2\left(D_{b}^{0,1}\tilde{\psi}-(-1)^r\tilde{\psi}D_{a}^{0,1}\right).\]

We then have a faithful embedding $\mathcal{I}:DG(\mathcal{X}_{\boldsymbol{\Sigma}})\rightarrow \mathcal{V}$ by restricting line bundle on $\mathcal{X}_{\boldsymbol{\Sigma}}$ to $\check{Y}$.
We define $\mathcal{V}'$ to be the image $\mathcal{I}(DG(\mathcal{X}_{\boldsymbol{\Sigma}}))$ and regard $DG(\mathcal{X}_{\boldsymbol{\Sigma}})$ as $\mathcal{V}'$.
    
\subsection{The category $Mo(P)$ of weighted Morse homotopy}
We recall the category $Mo(P)$ of weighted Morse homotopy on the polytope $P$, which is proposed by \cite{futaki2021homological} as a generalization of the category of weighted Morse homotopy given by \cite{kontsevich2001homological} to the case when the base space has boundaries and critical points may be degenerate.  
Let $P$ be the polytope of $\mathcal{X}_{\boldsymbol{\Sigma}}$, and $\pi:Y\rightarrow B=\text{Int}P$ the dual torus fibration of $\check{Y}$.
We consider $(B,g)$ as a Riemannian manifold induced by a Hessian structure given by section 4.1.

The objects of $Mo(P)$ are Lagrangian sections $\underline{s}$ of $\pi:
Y\rightarrow B$ corresponding to objects of $DG(\mathcal{X})$ described in the previous subsection. 
Namely, we take Lagrangian sections as objects which are SYZ mirror to holomorphic line bundles on $\check{Y}$ that can be extended to holomorphic line bundles on the whole $\mathcal{X}_{\boldsymbol{\Sigma}}$. 
Hereafter, we identify $\underline{s}$ with its graph, which we denote by $L$, and 
 assume that any two objects $L,L'$ intersect cleanly. 
By cleanly we mean that 
there exists an open set $\tilde{B}\subset M_{\mathbb{R}}$ such that $\overline{B}=P\subset \tilde{B}$ and $L,L'$ can be extended to graphs of smooth sections over $\tilde{B}$ so that they intersect cleanly.
\footnote{In \cite{nakanishi2024homological}, a generalized version of $Mo(P)$ is introduced, and the objects are taken so that they intersect generically cleanly.}

Let $(L, L')$ be an ordered pair of objects in $Mo(P)$. 
Let $V$ be a connected component of $\pi(L\cap L')\subset P$.
For $v\in V$, denote by $S_v\subset \tilde{B}$ the stable manifold of $v$ of the flow of the gradient vector field given by the difference of the extended graphs of $L'-L$, which can be expressed as $-\text{grad}(f_s-f_{s'})$ when restricted to $B$. Here, $f_s,f_{s'}$ denote the corresponding potential of lifts $s,s'$ to the connected component $V$.  Note that we have 
$-\text{grad}(f_s-f_{s'})=-\sum_{k}(\sum_{l}\frac{\partial(f_s-f_{s'})}{\partial x^{l}}\check{g}^{kl})\frac{\partial}{\partial  x^{k}}=-\sum_{k}\frac{\partial (f_s-f_{s'})}{\partial \check{x}_k}\frac{\partial}{\partial  x^{k}}=-\sum_{k}(s-s')\frac{\partial}{\partial x^{k}}$,
where $(\check{g}^{kl})=(g_{kl})^{-1}$. 
Now, for each $L,L'\in Mo(P)$, the space of morphisms is the $\mathbb{Z}$-graded vector space given as 
\[Mo(P)(L,L')=\bigoplus_{V\subset \pi(L\cap L')}\mathbb{C}\cdot V\]
where $V$ runs over all connected components of $\pi(L\cap L')\subset P$ that satisfy the following property:
 there exists a point $v\in V$ such that $v$ belongs to the interior of $S_v\cap P$, where we regard $S_v\cap P$ as a topological subspace of $S_v$. 
 \footnote{We consider the Morse cohomology degree instead of the Morse homology degree.}
The $\mathbb{Z}$-grading $|V|$ of a generator $V$ is defined by $|V|=\dim S_v$, which does not depend on the choice of $v\in V$.
We denote by $Mo^{r}(P)(L,L')$ the $r$-th graded piece of $Mo(P)(L,L')$. 

We explain $A_{\infty}$ products only for $\mathfrak{m}_2$, which is the composition of morphisms, for the following reasons. 
In our examples of weighted projective spaces, we see that the category $Mo(P)$ is minimal, i.e., the differential $\mathfrak{m}_1$ is trivial. 
Moreover, we can concretely take a finite set $\mathcal{E}$ of objects such that $\mathcal{E}$ forms the full strongly exceptional collection of $Tr(Mo_{\mathcal{E}}(P))$. 
In this situation, we see that in the full subcategory $Mo_{\mathcal{E}}(P)$ the higher products $\mathfrak{m}_k,k\geq3$ are all trivial by degree counting, so that $Mo_{\mathcal{E}}(P)$ is a DG category for a chosen $\mathcal{E}$. 
So it is sufficient to compute $\mathfrak{m}_2$ only to compute $Tr(Mo_{\mathcal{E}}(P))$. 
For more details and comments, see \cite{futaki2021homological}, section 4.5.

Let us consider a $3$-tuple  $(L_1,L_2,L_3)$. 
Let $V_{12}\in Mo(P)(L_1,L_2)$ and $V_{23}\in Mo(P)(L_2,L_3)$. 
If $V_{i(i+1)}$ for $i=1,2$ are determined by intersections $\text{graph}(df_i)\cap\text{graph}(df_{i+1})$, then we can take a connected component $V_{13}$ determined by the intersection $\text{graph}(df_{1})\cap\text{graph}(df_{3})$. Define $\mathcal{GT}(v_{12},v_{23};v_{13})$ to be the set of gradient trees starting at $v_{12}\in V_{12}$ and $v_{23}\in V_{23}$ and ending at $v_{13}\in V_{13}$. 
Here, a gradient 2-tree $\gamma\in \mathcal{GT}(v_{12},v_{23};v_{13})$ is a continuous map $\gamma:T\rightarrow P$ with a rooted trivalent tree $T$. Regarding $T$ as a planer tree, the leaf external vertices are mapped to $v_{12},v_{23}$ in order and the root external vertex is mapped to $v_{13}$. Moreover, for each edge $e$ of $T$, $\gamma|_{e}$ is a gradient trajectory of the corresponding gradient vector field, and these gradient trajectories meet at the internal vertex. 
More precisely, $\gamma$ can be identified with $(l_{12}(t),l_{23}(t),l_{13}(t))$ where $l_{i(i+1)}(t),i=1,2$ denote the gradient trajectories corresponding to $-\text{grad}(f_{i}-f_{i+1})$ with $\lim_{t\to -\infty}l_{i(i+1)}(t)=v_{i(i+1)}$ and $l_{13}(t)$ denote the gradient trajectories corresponding to $-\text{grad}(f_{1}-f_{3})$ with $\lim_{t\to \infty}l_{13}(t)=v_{13}$ which satisfy  $l_{12}(0)=l_{23}(0)=l_{13}(0)$. (See Figure \ref{gradient tree}.)
\begin{figure}[htbp]
    \centering 
    \begin{tikzpicture}
        \begin{feynhand}
            \vertex[dot] [label=$v_{12}$](l1) at (-1,1){}; \vertex (0) at (0,0);
            \propag[fer] (l1) to [quarter left,looseness=0.5, edge label'=$-\text{grad}(f_1-f_2)$] (0);
            \vertex[dot] [label=$v_{23}$](l2) at (1,1.5){};
            \propag[fer] (l2) to [quarter right,looseness=0.5, edge label=$-\text{grad}(f_2-f_3)$] (0);
            \vertex[dot] (r) at (0,-1){};
            \propag[fer] (0) to [edge label=$-\text{grad}(f_1-f_3)$](r);
        \end{feynhand}
        \draw (0,-1)node[below]{$v_{13}$};
    \end{tikzpicture}
    \caption{\mbox{A gradient tree $\gamma\in \mathcal{GT}(v_{12},v_{23};v_{13})$}}
    \label{gradient tree}
\end{figure}

Define 
\[\mathcal{GT}(V_{12},V_{23};V_{13}):=\bigcup_{(v_{12},v_{23};v_{13})\in V_{12}\times V_{23}\times V_{13}}\mathcal{GT}(v_{12},v_{23};v_{13})\]
and 
\[\mathcal{HGT}(V_{12},V_{23};V_{13}):=\mathcal{GT}(V_{12},V_{23};V_{13})/\text{smooth homotopy},\]
where two $\gamma,\gamma'$ are $C^{\infty}$-homotopic to each other if $\gamma$ is homotopic to $\gamma'$ so that the restriction $\gamma|_e$ is $C^{\infty}$-homotopic to $\gamma'|_{e}$ for each edge $e$ of $T$. 
We assume that the functions $f_i-f_{j}$ assigned to each edge are (Bott-)Morse-Smale, and the (un)stable manifolds of the vertices intersect transversely. 
Note that $\mathcal{HGT}(V_{12},V_{23};V_{13})$ is a finite set when $|V_{12}|+|V_{23}|=|V_{13}|$. Then, we define the composition $\mathfrak{m}_2$ by
\begin{align*}
    \mathfrak{m}_2:Mo(P)(L_1,L_2)\otimes Mo(P)(L_2,L_3)\rightarrow Mo(P)(L_1,L_3), \\
    \mathfrak{m}_2(V_{12}, V_{23})
    =\sum_{|V_{13}|=|V_{12}|+|V_{23}|}\sum_{[\gamma]\in \mathcal{HGT}(V_{12},V_{23};V_{13})}e^{-A(\gamma)}V_{13}
\end{align*}
where $A(\gamma)$ denotes the symplectic area of disk in $\pi^{-1}(\gamma(T))$. The weight $e^{-A(\gamma)}$ is invariant with respect to a $C^{\infty}$-homotopy.

\section{Homological mirror symmetry for weighted projective spaces} 
    In this section, we discuss the homological mirror symmetry in the sense of \cite{futaki2021homological} for the weighted projective space $\mathbb{P}(q_0,\ldots,q_n)$ (or $\mathbb{P}(Q)$ for short) with $\gcd(q_0,\ldots,q_n)=1$. 
    By \cite{auroux2006mirror,borisov2009conjecture}, it is known that the ordered set 
    \[\mathcal{E}:=(\mathcal{O}(q),\ldots,\mathcal{O}(q+\sum_{i=0}^nq_i+1))\] 
    (for any fixed integer $q$) forms a full strongly exceptional collection of $D^b(coh(\mathbb{P}(Q)))\simeq  Tr(DG_{\mathcal{E}}(\mathbb{P}(Q)))$.  Here, $DG_{\mathcal{E}}(\mathbb{P}(Q))$ is the full subcategory of $DG(\mathbb{P}(Q))$ consisting of holomorphic line bundles in $\mathcal{E}$ and Tr is the Bondal-Kapranov-Kontsevich's construction of triangulated categories \cite{bondal1990enhanced,kontsevich1995homological}. 
    In terms of the SYZ torus fibrations $Y\rightarrow B,\check{Y}\rightarrow B$, it is natural to consider DG category $\mathcal{V}=\mathcal{V}(\check{Y})$ on the complex side. 
    However, what we would like to discuss is a homological mirror symmetry for the whole $\mathcal{X}_{\boldsymbol{\Sigma}}\rightarrow P=\overline{B}$ and its dual. 
    Namely, we should consider the data of the whole space  by adding toric divisors or by adding the boundaries of the polytope. 
    For this purpose, we rather consider a faithful embedding $\mathcal{V}'_{\mathcal{E}}:=\mathcal{I}(DG_{\mathcal{E}}(\mathbb{P}(q_0,\ldots,q_n)))$ on the complex side. 
    On symplectic side, we consider the category $Mo(P)$ of weighted Morse homotopy on the polytope $P$ of $\mathbb{P}(q_0,\ldots,q_n)$ and in particular the full subcategory $Mo_{\mathcal{E}}(P)$.  Here, we abusively denote by $\mathcal{E}$ the ordered set of Lagrangian sections that are SYZ mirror to holomorphic line bundles in $\mathcal{E}=(\mathcal{O}(q),\ldots,\mathcal{O}(q+\sum_{i=0}^nq_i-1))$.
    Then, our main theorem is stated as follows. We note that the correspondence of objects is given by the SYZ transformation. 
\begin{thm}\label{maintheorem}
    For the weighted projective space $\mathbb{P}(q_0,\ldots,q_n)$ with $\gcd(q_0,\ldots,q_n)=1$,  
 there exists a DG-equivalence \[\iota:Mo_{\mathcal{E}}(P)\xrightarrow{\simeq} \mathcal{V}'_{\mathcal{E}}\]
 such that for any generator $V_{ab;K_{ab}}\in Mo_{\mathcal{E}}(P)(L_a,L_b)$ 
\begin{itemize}
    \item $\iota(V_{ab;K_{ab}})\in\mathcal{V}^{'}_{\mathcal{E}}(\iota(L_a),\iota(L_b))$ is a continuous function on $B$ and extends to continuously on $P=\overline{B}$,
    \item we have 
            \[\max_{x\in P}|\iota(V_{ab;K_{ab}})|=1,
            \quad \{x\in P\mid |\iota(V_{ab;K_{ab}})|=1\}=V_{ab;K_{ab}}.\]
\end{itemize}
\end{thm}

Since we have the DG isomorphism 
\[DG_{\mathcal{E}}
(\mathbb{P}(q_0,\ldots,q_n))
\simeq\mathcal{I}(DG_{\mathcal{E}}(\mathbb{P}(q_0,\ldots,q_n)))
=\mathcal{V}_{\mathcal{E}}^{'},\] 
we obtain the following, as a version of homological mirror symmetry for weighted projective spaces.

\begin{cor}
    There exists a DG-equivalence
        \[Mo_{\mathcal{E}}(P)\simeq DG_{\mathcal{E}}(\mathbb{P}(q_0,\ldots,q_n)).\]
\end{cor}

\begin{cor}\label{HMSforWPS}
    There exists an equivalence of  triangulated categories 
    \[Tr(Mo_{\mathcal{E}}(P))\simeq D^{b}(coh(\mathbb{P}(q_0,\ldots,q_n))),\]
    where $Tr$ denotes the construction of triangulated categories from $A_{\infty}$-categories by Bondal-Kapranov\cite{bondal1990enhanced} and Kontsevich\cite{kontsevich1995homological}.
\end{cor}

In the first two subsections 6.1, 6.2, we compute the cohomologies of the DG categories $DG(\mathbb{P}_{\mathcal{E}}(q_0,\ldots,q_n))$ and $\mathcal{V}'_{\mathcal{E}}$. 
In subsection 6.3, we compute $Mo_{\mathcal{E}}(P)$, which turns out to be minimal, i.e., have zero differential and a DG category. In the last subsection, we show Theorem \ref{maintheorem} by constructing the DG-equivalence $\iota$ explicitly. 

\subsection{Cohomologies of $DG_{\mathcal{E}}(\mathbb{P}(q_0,\ldots,q_n))$}

        We set $DG(\mathbb{P}(Q))$ as the DG category consisting of holomorphic orbifold line bundles $\mathcal{O}(a),\ a\in\mathbb{Z}$, where each $\mathcal{O}(a)$ is associated with a connection $d+A_a$ given by \eqref{connection form}. 
        Since the space $DG(\mathbb{P}(Q))(\mathcal{O}(a),\mathcal{O}(b))$ of morphism is defined as the Dolbeault resolution of the space of holomorphic morphisms $\text{Hom}(\mathcal{O}(a),\mathcal{O}(b))$, note that we have   
    \begin{align*}H^{r}(DG(\mathbb{P}(Q))(\mathcal{O}(a),\mathcal{O}(b)))&\simeq H^{r}(DG(\mathbb{P}(Q))(\mathcal{O},\mathcal{O}(b-a)))\\ &\simeq H^{r}(\mathbb{P}(Q),\mathcal{O}(b-a)).\end{align*}

    By Proposition 2.7 in \cite{auroux2008mirror}, we have 
    \begin{align}
        H^r(DG_{\mathcal{E}}(\mathbb{P}(Q))(\mathcal{O}(a),\mathcal{O}(b)))=
        \begin{cases}
            S_{b-a} \quad \text{for} \ r=0,\ a\leq b,\\
            0 \quad \text{otherwise},            
        \end{cases}
    \end{align}
    where $S$ denotes the graded algebra defined by the polynomial algebra $\mathbb{C}[z_0,\ldots,z_n]$ graded by $\text{deg}(z_i)=q_i$ for $i=0,\ldots,n$ and $S_r$ denotes the $r$-th graded piece of $S=\bigoplus_{r=0}^{\infty}S_r$.
If we use the coordinates $w_{i0},\ldots,\widehat{w_{ii}},\ldots, w_{in}$ for $\widetilde{U_i}\simeq \mathbb{C}^n$, 
each generator of $H^{0}(DG(\mathbb{P}(Q))(\mathcal{O}(a),\mathcal{O}(b)))$, $a<b$, can be locally expressed as 
\begin{align} 
    \widetilde{\psi}_{ab;K_{ab}}:=(w_{i0})^{k_0}\cdots (w_{i,i-1})^{k_{i-1}}(w_{i,i+1})^{k_{i+1}}\cdots (w_{in})^{k_n}, 
\end{align}
    where $K_{ab}=(k_0,\ldots,k_n)\in \mathbb{Z}_{\geq0}^{n+1}$ satisfies $\sum_{j=0}^{n}q_jk_j=b-a$. 

\subsection{Cohomologies of $\mathcal{V}'_{\mathcal{E}}$}  

        We set $\mathcal{V}=\mathcal{V}(\check{Y})$ as the DG category consisting of the restricted line bundles $\mathcal{O}(a)|_{\check{Y}},\ a\in\mathbb{Z}$, equipped with the twisted connections \eqref{twisted connection}. 
        Then, we consider the faithful embedding $\mathcal{I}: DG(\mathbb{P}(Q))\rightarrow \mathcal{V}$ and identify $DG(\mathbb{P}(Q))$ with its image 
            \[\mathcal{V}':=\mathcal{I}(DG(\mathbb{P}(Q))).\]
        We note that $\mathcal{V}'$ is non-full subcategory of $\mathcal{V}$ since morphisms of $\mathcal{V}=\mathcal{V}(\check{Y})$ are not required to be smooth on the toric divisors  $[(z_j=0)\cap Z/\mathbb{C}^{*}]$. 

        For $a<b$, the functor $\mathcal{I}$ maps each generator $[\widetilde{\psi}_{ab;K_{ab}}]$ of $H^{0}(DG(\mathbb{P}(Q))(\mathcal{O}(a),\mathcal{O}(b)))$ to each generator $[\psi_{ab;K_{ab}}]$ of $ H^{0}(\mathcal{V}^{'}(\mathcal{O}(a),\mathcal{O}(b)))$:
        \begin{align*}
            \psi_{ab;K_{ab}}&:=\Psi^{-1}_{b}\circ (\tilde{\psi}_{ab;K_{ab}})|_{\check{Y}}\circ \Psi_{a}\\
                            &=\Big(1+\sum_{j\neq i}e^{\frac{2q_0\cdots q_n}{q_j}\check{x}_{ij}}\Big)^{\frac{-b+a}{q_0\cdots q_n}}e^{K_{ab}^{(i)}\cdot{}^{t}\check{x}_{(i)}+\sqrt{-1}K_{ab}^{(i)}\cdot{}^{t}\check{y}_{(i)}}.
        \end{align*}
        If we use the dual coordinates $x^{(i)}=(x^{i1},\ldots,\widehat{x^{ii}},\ldots,x^{in})$, each $\psi_{ab;K_{ab}}$ is locally expressed as 
        \begin{align}\label{generators psi}
            \psi_{ab;K_{ab}}
            &=\left(\frac{2q_0\cdots q_n-q_0x^{i0}-\cdots-q_nx^{in}}{2q_0\cdots q_n}\right)^{\frac{q_ik_i}{2q_0\cdots q_n}}  \\ 
            &\qquad \times
            \left(\frac{q_0x^{i0}}{2q_0\cdots q_n}
            \right)^{\frac{q_0k_0}{2q_0\cdots q_n}}\cdots 
            \left(\frac{q_nx^{in}}{2q_0\cdots q_n}
            \right)^{\frac{q_nk_n}{2q_0\cdots q_n}}
            e^{\sqrt{-1}K_{ab}^{(i)}\cdot{}^{t}\check{y}_{(i)}}.
            \nonumber 
        \end{align}
        By this expression, we see that  each $\psi_{ab;K_{ab}}$ can be extended continuously on the whole polytope $P$.

\subsection{The Category $Mo(P)$ and cohomologies of $Mo_{\mathcal{E}}(P)$}
The objects of $Mo(P)$, where $P$ is the polytope of the weighted projective space $\mathbb{P}(q_0,\ldots,q_n)$, 
are the Lagrangian sections $\underline{s_a}$ of the dual torus fibration $\pi:Y\rightarrow B$ obtained in subsection 4.4.
Let us compute the space $Mo(P)(L_a,L_b)$ of morphisms. 
Firstly, in order to list up all the candidates of generators of morphisms, 
 we consider intersections of the graphs of lifts $y=s_{a;K_a}(x)$ and $y=s_{b;K_b}(x)$ of Lagrangian sections. Here, we treat them as continuously extended ones over $P$. 
The connected components of $\pi(L_a\cap L_b)\subset P$ are obtained by solving 
\[s_{b;K_b}(x)-s_{a;K_a}(x)=0,\quad x\in P.\]
    This is locally expressed as 
    \[2\pi\frac{b-a}{2q_0\cdots q_n}x^{(i)}-2\pi(K_{b}^{(i)}-K_{a}^{(i)})=0,\quad x^{(i)}\in P_{\sigma_i}.\]
    We note that the polytope $P_{\sigma_i}$ has an expression as a convex hull of the finite set $\{0,v^{i0},\ldots,\widehat{v^{ii}},\ldots,v^{in}\}$, where $v^{ik}\in \mathbb{R}^n,\ k\neq i$, is defined by $x^{il}(v^{ik})=\delta_{kl}\cdot \frac{2q_0\cdots q_n}{q_k},\ l\neq i$. Here, $\delta_{kl}$ is the Kronecker delta.  
    Namely, \[P_{\sigma_i}=\bigg\{\sum_{j\neq i}t_jv^{ij}\mid t_j\in\mathbb{R}_{\geq 0}\ \text{for}\ j\neq i, \sum_{j\neq i}t_j\leq 1\bigg\}.\] 
    If $a=b$, we immediately see that we have $P$ as the only connected component of $\pi(L_a\cap L_b)$, which is sometimes denoted by $V_{aa;(0,0,0)}$.  For $a\neq b$, we have the following.
\begin{lem}\label{lemintersection-MH}
	Let $a\neq b$. 
    Then, each connected component of $\pi(L_a\cap L_b)$ consists of a point $v_{ab;K_{ab}}$ given by
	\[ x^{(i)}(v_{ab;K_{ab}})= \frac{1}{|b-a|}\left\{		q_0k_{0}\begin{pmatrix}\frac{2q_0\cdots q_n}{q_0}\\ 0 \\  \vdots \\ \\ 0  \end{pmatrix}
	+\cdots+q_ik_{i}\begin{pmatrix}0\\ \\ \vdots\\ \\  0\end{pmatrix}
    +\cdots+q_nk_{n}\begin{pmatrix}0\\ \\\vdots  \\0 \\ \frac{2q_0\cdots q_n}{q_n}\end{pmatrix} 
	\right\}
	,\]
	where $K_{ab}=(k_{0},\ldots,k_{n})\in\mathbb{Z}_{\geq 0}^{n+1}$ satisfies $\sum_{j=0}^{n}q_{j}k_{j}=|b-a|$, i.e., $v_{ab;K_{ab}}=\sum_{l\neq i}\frac{q_lk_l}{|b-a|}v^{il}$.
\end{lem}
We write $V_{ab;K_{ab}}=\{v_{ab;K_{ab}}\}$.
We next discuss when these connected components form generators of $Mo(P)(L_a,L_b)$. 
\begin{lem}
    Let $a<b$ (resp.$\ a=b$). Then, each connected component $V_{ab;K_{ab}}$ (resp.$\ P$) forms a generator of $Mo(P)(L_a,L_b)$ (resp.$\ Mo(P)(L_a,L_a)$) of degree zero.
\end{lem}

\begin{proof}
    The gradient vector field associated to $V_{ab;K_{ab}}$ is of the form
    \begin{align}\label{gradvectfield-stablemfd}
        2\pi\left(\frac{b-a}{2q_0\cdots q_n}x^{i0}-k_0\right)\frac{\partial }{\partial x^{i0}}+\cdots+2\pi\left(\frac{b-a}{2q_0\cdots q_n}x^{in}-k_n \right)\frac{\partial}{\partial x^{in}}.
    \end{align}
    {F}or $a<b$, we see that the corresponding stable manifold $S_{v_{ab;K_{ab}}}$ is $V_{ab;K_{ab}}=\{v_{ab;K_{ab}}\}$ itself. Therefore, $V_{ab;K_{ab}}$ is a generator of degree zero. For $a=b$, the gradient vector field associated to $P$ is $0$. Then, for any point $p\in P$, the corresponding stable manifold is $\{p\}$ itself, and $P$ forms a generator of degree zero.
\end{proof}

 If $a>b$, then the corresponding gradient vector field is the opposite sign of (\ref{gradvectfield-stablemfd}), which implies $S_{v_{ab;K_{ab}}}\cap P=P$. Therefore, each $V_{ab;K_{ab}},\ a>b$, does not form a generator unless $v_{ab;K_{ab}}$ belongs to the interior of $P\subset \mathbb{R}^n$.  
Now, for any positive integer $R\in\mathbb{Z}_{>0}$, let 
\[\mathcal{E}_{R}:=(L_q,\ldots,L_{q+R})\] 
be an (ordered) finite set of Lagrangian sections for any fixed integer $q$. Then, the following proposition allows us to determine whether $V_{ab;K_ab},\ a>b$, forms a generator. 

\begin{prop}
    Every connected component $V_{ab;K_{ab}}\subset P$ of $\pi(L_a\cap L_b)$ for any $L_a,L_b\in\mathcal{E}_R$ with $L_a\neq L_b$ belongs to the boundary $\partial P$ if and only if $R\leq \sum_{j=0}^{n}q_j-1$.
\end{prop}

\begin{proof}
    For $L_a,L_b\in\mathcal{E}_R,L_a\neq L_b$, suppose that there exists a connected component $V_{ab;K_{ab}}$ of $\pi(L_a\cap L_b)$ such that $v_{ab;K_{ab}}=\sum_{j\neq i}\frac{q_{j}k_j}{|b-a|}v^{ij}$ is in the interior of $P$. Since $\mathrm{Int}(P_{\sigma_i})=\{\sum_{j\neq i}t_jv^{ij}\mid t_j>0\ (j\neq i),\sum_{j\neq i}t_j<1\}$, we see that  $k_j>0$ for all $j=0,\ldots,n$. 
    Then $|b-a|$, which is equal to $\sum_{j=0}^{n}q_jk_j$, must be greater than or equal to $\sum_{j=0}^{n}q_j$.  
    Conversely, let $R> \sum_{j=0}^{n}q_j-1$. Then, for $L_a,L_b\in \mathcal{E}_R$ with $|b-a|=\sum_{j=0}^{n}q_j$, there is a connected component $V_{ab;(1,\ldots,1)}$ of $\pi(L_a\cap L_b)$ consisting of a point $v_{ab;(1,\ldots,1)}=\sum_{j\neq i}\frac{q_{j}}{|b-a|}v^{ij}\in \text{Int}(P_{\sigma_i})$.
\end{proof}

The above proposition implies the following.

\begin{cor}
    The ordered set $\mathcal{E}_{R}$ forms a full strongly exceptional collection in $Tr(Mo_{\mathcal{E}_{R}}(P))$ if and only if $R\leq \sum_{j=0}^nq_j-1$.
\end{cor}

Namely, we see that $\mathcal{E}=\mathcal{E}_{\sum_{j=0}^nq_j-1}$ is the longest sequence of Lagrangian sections of the form of $\mathcal{E}_R$ such that $\mathcal{E}_{R}$ forms a full strongly exceptional collection in $Tr(Mo_{\mathcal{E}_R}(P))$. 
By the discussions above, in the full subcategory $Mo_{\mathcal{E}}(P)$, the space of morphisms is explicitly given as follows: 
\begin{align} \label{complex-mop}
    Mo^{r}_{\mathcal{E}}(P)(L_a,L_b)
    =   \begin{cases}
            \mathbb{C}\cdot P \quad \text{for} \ a=b,\ r=0,\\
          \bigoplus_{\substack{K_{ab}=(k_0,\dots,k_n)\in\mathbb{Z}_{\geq 0}^{n+1}\\ \sum_{j=0}^{n}q_jk_j=b-a}}\mathbb{C}\cdot V_{ab;K_{ab}} \quad \text{for} \ a<b,\ r=0,\\
            0 \quad \text{otherwise}.
        \end{cases}
\end{align}

\begin{ex}\label{ex of generators of morphisms of Mo(P)}
In the case of $\mathbb{P}(1,1,2)$, the generators of $Mo^{0}_{\mathcal{E}}(P)(L_a,L_b) $, $0<b-a\leq3$, are illustrated as in Figure \ref{generators of morphisms of Mo(P) for P(1,1,2)}. Here, the smaller dots are the lattice $M_{\sigma_0}=\mathbb{Z}^2$. 

\begin{figure}[htbp]
    \centering
     \begin{minipage}[htbp]{0.48\columnwidth}
     \centering
     \begin{tikzpicture}[scale=1.1]
      \foreach \x in {0,...,4}
        \foreach \y in {0,...,2}\fill[black](\x,\y)circle(0.02); 
        \draw (0,0)--(4,0)--(0,2)--cycle;
        \fill[black](0,0)circle(0.075);
            \draw(0,0)node[below]{$v_{ab;(1,0,0)}$};
        \fill[black](4,0)circle(0.075);
            \draw(4,0)node[below]{$v_{ab;(0,1,0)}$};
    \end{tikzpicture}
    \caption*{$b-a=1$}
    \end{minipage}
    \quad 
    \begin{minipage}[htbp]{0.48\columnwidth}
        \centering
        \begin{tikzpicture}[scale=1.1]
        \foreach \x in {0,...,4}
        \foreach \y in {0,...,2}\fill[black](\x,\y)circle(0.02); 
            \draw (0,0)--(4,0)--(0,2)--cycle;    
            \fill[black](0,0)circle(0.075);
                \draw(0,0)node[below]{$v_{ab;(2,0,0)}$};
            \fill[black](2,0)circle(0.075);
                \draw(2,0)node[below]{$v_{ab;(1,1,0)}$};
            \fill[black](4,0)circle(0.075);
                \draw(4,0)node[below]{$v_{ab;(0,2,0)}$};
            \fill[black](0,2)circle(0.075);
                \draw(0,2)node[right]{$v_{ab;(0,0,1)}$};
        \end{tikzpicture}
        \caption*{$b-a=2$}        
    \end{minipage}
    \quad
    \begin{minipage}[htbp]{0.48\columnwidth}
    \centering
        \begin{tikzpicture}[scale=1.1]
        \foreach \x in {0,...,4}
        \foreach \y in {0,...,2}\fill[black](\x,\y)circle(0.02); 
            \draw (0,0)--(4,0)--(0,2)--cycle;    
            \fill[black](0,0)circle(0.075);
                \draw(0,0)node[below]{$v_{ab;(3,0,0)}$};
            \fill[black](4/3,0)circle(0.075);
                \draw(4/3,0)node[below]{$v_{ab;(2,1,0)}$};
            \fill[black](8/3,0)circle(0.075);
                \draw(8/3,0)node[below]{$v_{ab;(1,2,0)}$};
            \fill[black](4,0)circle(0.075);
                \draw(4,0)node[below]{$v_{ab;(0,3,0)}$};
            \fill[black](0,4/3)circle(0.075);
                \draw(0,4/3)node[left]{$v_{ab;(1,0,1)}$};
            \fill[black](4/3,4/3)circle(0.075);
                \draw(4/3,4/3)node[right]{$v_{ab;(0,1,1)}$};
        \end{tikzpicture}
        \caption*{$b-a=3$}       
    \end{minipage}
    \caption{The generators of $Mo^{0}_{\mathcal{E}}(P)(L_a,L_b)$ for $\mathbb{P}(1,1,2)$}
    \label{generators of morphisms of Mo(P) for P(1,1,2)}
\end{figure}
\end{ex}
    
 Next, we consider the $A_{\infty}$ structure of $Mo_{\mathcal{E}}(P)(L_a,L_b)$. 
 For degree reasons, we immediately see that 
 $\mathfrak{m}_1=0$ and $\mathfrak{m}_k=0,k\geq3$.
 Hence, $Mo_{\mathcal{E}}(P)$ is minimal and forms a DG category.
It remains to calculate $\mathfrak{m}_2$.   
Let $a< b < c$.  Let $V_{ab;K_{ab}}\in Mo_{\mathcal{E}}(P)(L_a,L_b)$ and let $V_{bc;K_{bc}}\in Mo_{\mathcal{E}}(P)(L_b,L_c)$. 
We can take potential functions $f_{a;K_a},f_{b;K_b},f_{c;K_c}$ for lifted Lagrangian sections $s_{a;K_a}, s_{b;K_b},s_{c;K_c}$ 
such that $-(s_{a;K_a}-s_{b;K_b})(v_{ab;K_{ab}})=0$ and $-(s_{b;K_b}-s_{c;K_c})(v_{bc;K_{bc}})=0$, where $K_b-K_a=K_{ab}$ and $K_{c}-K_b=K_{bc}$. Here, we recall potential means that $\sum_{k\neq i}s_{a;K_a}^{ik}d\check{x}_{ik}=df_{a;K_a}^{(i)}$. 
Let us associate to $V_{ab;K_{ab}}$ a potential function $f_{ab;K_{ab}}$ on $P$ which is uniquely defined by 
\begin{align}\label{fab}
    d(f_{b;K_b}-f_{a;K_a})=df_{ab;K_{ab}},\quad f_{ab;K_{ab}}(v_{ab;K_{ab}})=0.
    \end{align}
Similarly, we associate to $V_{bc;K_{bc}}$ a potential function $f_{bc;K_{bc}}$.

Now, consider the intersection of graphs of $s_{a;K_a}$ and $s_{c;K_c}$, and we have \[V_{ac;K_{ab}+K_{bc}}=\{v_{ac;K_{ab}+K_{bc}}\}\in Mo_{\mathcal{E}}(P)(L_a,L_c).\]
 This $v_{ac;K_{ab}+K_{bc}}$ is a point on $\partial P$ dividing the line segment connecting $v_{ab;K_{ab}}$ and $v_{bc;K_{bc}}$ in the ratio $c-b:b-a$.
Then, consider the set of gradient trees $\mathcal{GT}(v_{ab;K_{ab}},v_{bc;K_{bc}};v_{ac;K_{ab}+K_{bc}})$. 
Note that we now have \[U_{v_{ab;K_{ab}}}\cap U_{v_{bc;K_{bc}}}\cap S_{v_{ac;K_{ab}+K_{bc}}}=\{v_{ac;K_{ab}+K_{bc}}\},\] where $U_v$ denotes the corresponding unstable manifold. 
Therefore, there exists the unique gradient tree $\gamma$ up to smooth homotopy such that
\begin{itemize}
    \item the gradient trajectory of $-\mathrm{grad}(f_{a;K_a}-f_{b;K_b})=\mathrm{grad}f_{ab;K_{ab}}$ starting from $v_{ab;K_{ab}}$ goes straight,
    \item the gradient trajectory of $-\mathrm{grad}(f_{b;K_b}-f_{c;K_c})=\mathrm{grad}f_{bc;K_{bc}}$ starting from $v_{bc;K_{bc}}$ goes straight,
    \item the gradient trajectory of $-\mathrm{grad}(f_{a;K_a}-f_{c;K_c})=\mathrm{grad}(f_{ab;K_{ab}}+f_{bc;K_{bc}})$ ending at $v_{ac;K_{ab}+K_{bc}}$ stays at $v_{ac;K_{ab}+K_{bc}}$.
\end{itemize}
This means that they should meet at $v_{ac;K_{ab}+K_{bc}}$.

\begin{ex}
We continue with Example \ref{ex of generators of morphisms of Mo(P)}.
For generators $v_{ab;(0,1,0)}\ (b-a=1)$ and  $v_{bc;(0,0,1)}\ (c-b=2)$, the following Figure \ref{figure of a gradient tree} shows  the gradient tree $\gamma\in \mathcal{GT}(v_{ab;(0,1,0)},v_{bc;(0,0,1)};v_{ac;(0,1,0)+(0,0,1)})$.
\begin{figure}[htbp]
    \centering
    \begin{tikzpicture}[scale=1.5]
        \draw (0,0)--(4,0)--(0,2)--cycle;
         \fill[black](4,0)circle(0.075);
            \draw(4,0)node[below left]{$v_{ab;(0,1,0)}$};   
        \fill[black](0,2)circle(0.075);
                \draw(0,2)node[below left]{$v_{bc;(0,0,1)}$};
         \fill[black](4/3,4/3)circle(0.075);
                \draw(4/3,4/3)node[below left]{$v_{ac;(0,1,1)}$};
        \draw [very thick](4,0)--(0,2);
        \draw[<-](4/3+0.1,4/3)to[out=30,in=150](4/3+0.5,4/3)node[right]{$\text{grad}(f_{ab;(0,1,0)}+f_{bc;(0,0,1)})$};
         \begin{feynhand}
            \vertex (l1) at(4,0); \vertex (r) at (4/3,4/3); \vertex (l2) at (0,2);
            \propag[fer] (l1) to [edge label'=$\text{grad}f_{ab;(0,1,0)}$](r);
            \propag[fer] (l2) to [edge label=$\text{grad}f_{bc;(0,0,1)}$](r);
         \end{feynhand}

    \end{tikzpicture}
    \caption{$\text{The image } \gamma(T)\subset P$, \\ where $\gamma \in\mathcal{GT}(v_{ab;(0,1,0)},v_{bc;(0,0,1)};v_{ac;(0,1,0)+(0,0,1)})$}
    \label{figure of a gradient tree}
\end{figure}
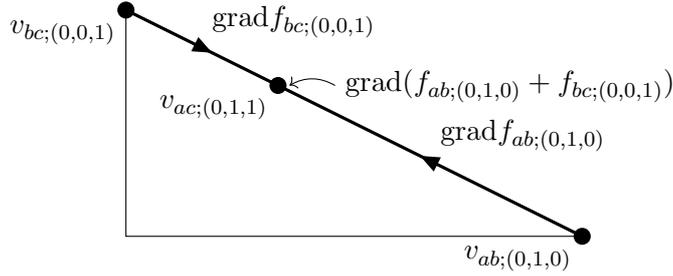
    
\end{ex}

Now the symplectic area $A(\gamma)$ turns out to be 
\[A(\gamma)=f_{ab;K_{ab}}(v_{ac;K_{ab}+K_{bc}})+f_{bc;K_{bc}}(v_{ac;K_{ab}+K_{bc}}).\]
Here, $f_{ab;K_{ab}}(v_{ac;K_{ab}+K_{bc}})$ is the symplectic area of the triangle disk enclosed by $s_{a;K_a}(\gamma(T))$, $s_{b;K_{b}}(\gamma (T))$ and $\pi^{-1}(v_{ac;K_{ab}+K_{bc}})$. Similarly, $f_{bc;K_{bc}}(v_{ac;K_{ab}+K_{bc}})$ is the symplectic area enclosed by $s_{c;K_c}(\gamma(T))$, $s_{b;K_{b}}(\gamma (T))$ and $\pi^{-1}(v_{ac;K_{ab}+K_{bc}})$. Thus, for $a<b<c$, we obtain 
\begin{align}\label{composition}
    \mathfrak{m}_{2}(V_{ab;K_{ab}},V_{bc;K_{bc}})=e^{-(f_{ab;K_{ab}}(v_{ac;K_{ab}+K_{bc}})+f_{bc;K_{bc}}(v_{ac;K_{ab}+K_{bc}}))}V_{ac;K_{ab}+K_{bc}}.
\end{align}

It remains to consider the cases $a=b<c$, $a<b=c$ or $a=b=c$. If $a=b$ (resp.$\ b=c$), then $V_{ab;K_{ab}}=P$ (resp.$\ V_{bc;K_{bc}}=P$). 
Since the gradient vector field associated to $P$ equals to zero, we see that the image $\gamma(T)$ shrinks to a point in all cases above. 
Thus, we see that $P$ forms the identity morphism with respect to $\mathfrak{m}_2$. 

As a byproduct of discussions above, we obtain the following.

\begin{prop}
    The image $\gamma(T)$ of any $\gamma\in \mathcal{GT}(V_{ab;K_{ab}},V_{bc;K_{bc}};V_{ac;K_{ac}})$ is always contained in the boundary $\partial P$ unless $a=b=c$. 
\end{prop}

\subsection{Construction of the DG-equivalence in the main theorem}        
In this subsection, we show our main theorem by constructing the DG-equivalence explicitly. 
We mention that the basic strategy that is originally proposed in \cite{futaki2021homological} in turn works well for our cases. 
Recall that a DG-equivalence is a DG functor which induces a category equivalence on the corresponding cohomology categories. 
\begin{lem}\label{quasi-isom}
   There exists a quasi-isomorphism $\iota$ of cochain complexes
    \[\iota:Mo_{\mathcal{E}}(P)(L_a,L_b)\rightarrow \mathcal{V}_{\mathcal{E}}^{'}(\mathcal{O}(a),\mathcal{O}(b))\]
    satisfying the two properties of Theorem \ref{maintheorem}.
\end{lem}
\begin{proof}  
    We first notice that both generators of $H^0(Mo_{\mathcal{E}}(P)(L_a,L_b))$ and $H^{0}(\mathcal{V}'_{\mathcal{E}}(\mathcal{O}(a),\mathcal{O}(b)))$ with $a\leq b$ come from $K_{ab}=(k_0,\ldots,k_n)\in \mathbb{Z}_{\geq 0}^{n+1}$ such that $\sum_{j=0}^nq_jk_j=b-a$.
    Furthermore, for $V_{ab;K_{ab}}$ and $\psi_{ab;K_{ab}}$ with $a<b$, we see that there exists a positive number $c_{ab;K_{ab}}$ such that 
    \[e^{-\frac{1}{2\pi}f_{ab;K_{ab}}+\sqrt{-1}K_{ab}\check{y}}=c_{ab;K_{ab}}\psi_{ab;K_{ab}},\]
    by comparing expressions of the potential $f_{ab;K_{ab}}$ defined as in \eqref{potential} and $\psi_{ab;K_{ab}}$ defined by \eqref{generators psi}. 
    We can take such $c_{ab;K_{ab}}$ so that 
    \[\max_{x\in P}|c_{ab;K_{ab}}\psi_{ab;K_{ab}}|=1.\]
    Let us rescale all $\psi_{ab;K_{ab}}$ and write $\mathbf{e}_{ab;K_{ab}}=c_{ab;K_{ab}}\psi_{ab;K_{ab}}$.
    We can take these $[\mathbf{e}_{ab;K_{ab}}]$'s as a base of $H^{0}(\mathcal{V}'_{\mathcal{E}}(\mathcal{O}(a),\mathcal{O}(b)))$ instead of $\psi_{ab;K_{ab}}$'s.
    Then, define the map $\iota: Mo_{\mathcal{E}}(L_a,L_b)\rightarrow \mathcal{V}'_{\mathcal{E}}(\mathcal{O}(a),\mathcal{O}(b)),a<b$, by setting
    \[\iota:V_{ab;K_{ab}}\longmapsto e^{-\frac{1}{2\pi}f_{ab;K_{ab}}+\sqrt{-1}K_{ab}\check{y}}= \mathbf{e}_{ab;K_{ab}}.\]
    This construction is valid for $a\leq b$ by taking $f_{aa;(0,\ldots,0)}=0$ and $\mathbf{e}_{aa;(0,\ldots,0)}=1$ for $a=b$. Additionally, consider $\iota$ to be the zero map for $a>b$. 
    Thus, we obtain the map $\iota: Mo_{\mathcal{E}}(L_a,L_b)\rightarrow \mathcal{V}'_{\mathcal{E}}(\mathcal{O}(a),\mathcal{O}(b))$, which turns out to be a quasi-isomorphism by construction.
    Moreover, by construction, it follows that for $\mathbf{e}_{ab;K_{ab}}=\iota(V_{ab;K_{ab}})\in \mathcal{V}'_{\mathcal{E}}(\mathcal{O}(a),\mathcal{O}(b))$
    \[\{x\in P\mid |\mathbf{e}_{ab;K_{ab}}|=1\}=V_{ab;K_{ab}}.\]
    Thus, our $\iota$ satisfies all the desired properties. 
\end{proof}

It remains to check the compatibility of product structures in order to complete Theorem \ref{maintheorem}. 

\begin{lem}\label{compatibility-composition}
    Let $\iota:Mo_{\mathcal{E}}(P)(L_a,L_b)\rightarrow \mathcal{V}'_{\mathcal{E}}(\mathcal{O}(a),\mathcal{O}(b))$ be the quasi-isomorphism constructed in the proof of Lemma \ref{quasi-isom}. 
    Then, we have
    \begin{align}\label{compaibility-composition equality}\iota(\mathfrak{m}_{2}(V_{ab;K_{ab}},V_{bc;K_{bc}}))=\mathbf{e}_{ab;K_{ab}}\cdot \mathbf{e}_{bc;K_{bc}}.
    \end{align}
\end{lem}
\begin{proof}
    Let $a<b<c$.
    By \eqref{composition} the left hand side of \eqref{compaibility-composition equality} is expressed as
        \[\iota(\mathfrak{m}_{2}(V_{ab;K_{ab}},V_{bc;K_{bc}}))=e^{-\frac{1}{2\pi}\left(f_{ab;K_{ab}}(v_{ac;K_{ab}})+f_{bc;K_{bc}}(v_{ac;K_{ac}})+f_{ac;K_{ac}}\right)
        +\sqrt{-1}K_{ac}\check{y}}, \]
        where $K_{ac}:=K_{ab}+K_{bc}$, 
        and the right hand side of \eqref{compaibility-composition equality} is expressed as 
        \[  \mathbf{e}_{ab;K_{ab}}\cdot\mathbf{e}_{bc;K_{bc}}=e^{-\frac{1}{2\pi}\left(f_{ab;K_{ab}}+f_{bc;K_{bc}}\right)+\sqrt{-1}K_{ac}\check{y}}.
        \]
        Notice that the difference between $f_{ac;K_{ac}=K_{ab}+K_{bc}}$  and $f_{ab;K_{ab}}+f_{bc;K_{bc}}$ is just a constant function, since $d(f_{ab;K_{ab}}+f_{bc;K_{bc}})=df_{ac;K_{ac}}$ (cf. \eqref{fab}). 
        In particular, we see that 
        \begin{align*}
            (f_{ab;K_{ab}}+f_{bc;K_{bc}}-f_{ac;K_{ac}})(x)\equiv(f_{ab;K_{ab}}+f_{bc;K_{bc}})(v_{ac;K_{ac}}),
        \end{align*}
        and hence obtain the compatibility \eqref{compaibility-composition equality}. 
        For the remaining cases $a=b<c,\ a<b=c$ or $a=b=c$, it immediately follows since the identity morphism $P$ is sent to $1$. 
\end{proof}

\begin{proof}[Proof of Theorem \ref{maintheorem}]
    Combining Lemma \ref{quasi-isom} and Lemma \ref{compatibility-composition}, we see that  $\iota$ 
    constructed in the proof of Lemma \ref{quasi-isom} 
    extends to a DG-equivalence. This completes the proof.
\end{proof}

\bibliographystyle{alpha}
\bibliography{hmsforWPSandMorsehomotopy/hms_for_wps_Morse_homotopy}

\end{document}